\documentclass{amsart}
\usepackage{amsmath}
\usepackage{amsopn}
\usepackage{amssymb}
\usepackage{amsthm}
\usepackage{mathrsfs}
\usepackage{stmaryrd}
\usepackage[french,english]{babel}
\usepackage[utf8]{inputenc}
\usepackage[T1]{fontenc}
\usepackage{lmodern}
\usepackage[all]{xy}
\usepackage{hyperref}
\usepackage{upref}
\usepackage{array}
\usepackage{color}
\usepackage{chngcntr}
\usepackage[titletoc]{appendix}
\usepackage{mathtools}

\newcommand{\FF}{\mathbb{F}}

\newcommand{\QQ}{\mathbb{Q}}

\def\XX{\mathfrak{X}}

\def\YY{\mathfrak{Y}}

\def\Rep{\mathbf{Rep}}

\def\FHom{\mathscr{H}om}

\def\Ga{\mathbb{G}_{a}}
\def\A1{\mathbb{A}^1}
\def\P1{\mathbb{P}^1}

\def\hH{\mathcal{H}}

\def\ccT{\mathcal{T}}

\def\Isod{\mathbf{Isoc}^{\dagger\dagger}}

\def\bQQ{\overline{\mathbb{Q}}}
\def\bQla{\overline{\mathbb{Q}}_{\lambda}}

\def\Ql{\mathbb{Q}_{\ell}}

\def\bQp{\overline{\mathbb{Q}}_p}
\def\bQpF{\overline{\mathbb{Q}}_{p,F}}

\def\Fet{\mathbf{Fet}}
\def\bGm{\mathbb{G}_m}
\def\HPi{\widehat{\Pi}}

\def\PhiMIC{\Phi\textnormal{-}\mathbf{MIC}}

\DeclareMathOperator{\LS}{\mathbf{LocSys}}
\DeclareMathOperator{\Fr}{Fr}
\DeclareMathOperator{\Pred}{Pro-red}
\DeclareMathOperator{\Pss}{Pro-ss}

\DeclareMathOperator{\IM}{Im}

\DeclareMathOperator{\et}{\textnormal{\'{e}t}}

\DeclareMathOperator{\MIC}{\mathbf{MIC}}

\DeclareMathOperator{\hol}{hol}
\DeclareMathOperator{\Hol}{Hol}

\DeclareMathOperator{\Gal}{Gal}
\DeclareMathOperator{\Ker}{Ker}

\DeclareMathOperator{\Hom}{Hom}

\DeclareMathOperator{\Ext}{Ext}

\DeclareMathOperator{\Spec}{Spec}

\DeclareMathOperator{\cont}{cont}

\DeclareMathOperator{\rH}{H}

\DeclareMathOperator{\rb}{b}

\DeclareMathOperator{\rD}{D}
\DeclareMathOperator{\rW}{W}

\DeclareMathOperator{\rR}{R}

\DeclareMathOperator{\id}{id}
\DeclareMathOperator{\rig}{rig}

\DeclareMathOperator{\Del}{\mathbf{Del}}

\DeclareMathOperator{\GL}{GL}

\DeclareMathOperator{\Sm}{Sm}

\DeclareMathOperator{\Vect}{Vec}

\DeclareMathOperator{\Aut}{Aut}

\DeclareMathOperator{\red}{red}
\DeclareMathOperator{\mot}{mot}
\DeclareMathOperator{\semis}{ss}

\DeclareMathOperator{\geo}{geo}
\DeclareMathOperator{\arith}{arith}

\DeclareMathOperator{\uni}{uni}

\DeclareMathOperator{\sep}{sep}

\DeclareMathOperator{\cst}{cst}

\DeclareMathOperator{\smooth}{smooth}
\DeclareMathOperator{\ev}{ev}
\DeclareMathOperator{\coev}{coev}

\DeclareMathOperator{\FIsoc}{\mathbf{F-Isoc}}
\DeclareMathOperator{\PhiIsoc}{\mathbf{\Phi-Isoc}}

\DeclareMathOperator{\FIsocd}{\mathbf{F-Isoc}^{\dagger}}
\DeclareMathOperator{\PhiIsocd}{\mathbf{\Phi-Isoc}^{\dagger}}
\DeclareMathOperator{\PhipIsocd}{\mathbf{\Phi}^{\prime}-\mathbf{Isoc}^{\dagger}}
\DeclareMathOperator{\Isocd}{\mathbf{Isoc}^{\dagger}}

\DeclareMathOperator{\Isoc}{\mathbf{Isoc}}

\DeclareMathOperator{\ur}{ur}
\DeclareMathOperator{\WD}{WD}

\newtheorem{theorem}{Theorem}[subsection]
\newtheorem{lemma}[theorem]{Lemma}
\newtheorem{prop}[theorem]{Proposition}
\newtheorem{cor}[theorem]{Corollary}

\theoremstyle{definition}

\newtheorem{remark}[theorem]{Remark}
\newtheorem{defn}[theorem]{Definition}

\newtheorem{secnumber}[theorem]{}

\numberwithin{equation}{subsection}
\numberwithin{equation}{theorem}

\newcolumntype{L}{>{$}l<{$}} 

\title{Drinfeld's lemma for $F$-isocrystals, II: Tannakian approach}
\author{Kiran S. Kedlaya, Daxin Xu}
\date{\today}
\thanks{The first author is supported by NSF grants DMS-1802161 and DMS-2053473 and the UC San Diego Warschawski Professorship. The second author is supported by National Natural Science Foundation of China Grant (Nos. 12222118, 12288201) and CAS Project for Young Scientists in Basic Research, Grant No. YSBR-033. D. X. is grateful to Xinwen Zhu for suggesting this question and helpful discussions.}

\AtEndDocument{\bigskip{\footnotesize

  \textsc{Kiran S. Kedlaya, Department of Mathematics, University of California San Diego, La Jolla, CA 92093, USA.} \par
  \textit{E-mail address}: \texttt{kedlaya@ucsd.edu} \par

  \textsc{Daxin Xu, Morningside Center of Mathematics and Hua Loo-Keng Key Laboratory of Mathematics, Academy of Mathematics and Systems Science, Chinese Academy of Sciences, Beijing 100190, China.}\par
  \textit{E-mail address}: \texttt{daxin.xu@amss.ac.cn} \par
}}

\begin{document}

\begin{abstract}
We prove a Tannakian form of Drinfeld's lemma for isocrystals on a variety over a finite field, equipped with actions of partial Frobenius operators. This provides an intermediate step towards transferring V. Lafforgue's work on the Langlands correspondence over function fields from $\ell$-adic to $p$-adic coefficients. 
We also discuss a motivic variant and a local variant of Drinfeld's lemma. 
\end{abstract}

\selectlanguage{english}
\maketitle

\section{Introduction}

This paper is a sequel of sorts to the first author's paper \cite{kedlaya-dri}, 
although there is no logical dependence between the two. Both papers concern themselves with
analogues of ``Drinfeld's lemma'' in \'etale cohomology, and in particular with corresponding statements in $p$-adic cohomology; however, there are some differences in scope and methodology which we highlight below.

\begin{secnumber} \label{sss:intro setting}
Drinfeld's lemma was first introduced by Drinfeld in his proof of the Langlands correspondence for $\GL_2$ over global function fields of characteristic $p>0$ \cite{Dri78}. 
Let us briefly review the result. Let $X_1,X_2$ be two connected schemes over $k=\mathbb{F}_p$. The scheme $X:=X_1\times_k X_2$ is equipped with two endomorphisms $F_{1}, F_{2}$, obtained by base changes of the absolute Frobenius on $X_1,X_2$, respectively. 
We consider the category $\mathcal{C}(X,\Phi)$ of objects $(T,F_{ \{1\}},F_{ \{2\}})$ consisting of a finite \'etale morphism $T\to X$ and isomorphisms $F_{ \{i\}}\colon T\times_{X,F_i}X \xrightarrow{\sim} T$ commuting with each other, whose composition is the relative Frobenius morphism $F_{T/X}$ of $T$ over $X$. 
This category is a Galois category and we denote by $\pi_1(X,\Phi,\overline{x})$ the Galois group defined by a geometric point $\overline{x}$ of $X$. 
Then Drinfeld's lemma says that the projection maps $X\to X_i$ induce an isomorphism of profinite groups:
\begin{equation} \label{eq:origin Drinfeld lemma}
	\pi_1(X,\Phi,\overline{x}) \xrightarrow{\sim} \pi_1(X_1,\overline{x})\times \pi_1(X_2,\overline{x}).
\end{equation}
(See \cite[Theorem~8.1.4]{lau} or \cite[Theorem~4.2.12]{kedlaya-aws}. For the key case where $X_2$ is a geometric point, see also \cite[Lemme~8.11]{lafforgue-v}.)

A closely related result is that any quasicompact 
open immersion which is stable under the $F_i$ is covered by products of open immersions into the $X_i$. 
(See \cite[Lemma~9.2.1]{lau} or \cite[Theorem~4.3.6]{kedlaya-aws}. For the key case where $X_2$ is a geometric point, see also \cite[Lemme~8.12]{lafforgue-v}.)
These two results together allow us to view lisse $\ell$-adic sheaves with ``partial Frobenius'' on $X$ as $\ell$-adic representations of $\pi_1(X_1,\overline{x})\times \pi_1(X_2,\overline{x})$, and to work with constructible sheaves with partial Frobenius on $X$ via stratifications coming from $X_1$ and $X_2$. 
We refer to \cite{lau} and \cite[Lecture 4]{kedlaya-aws} for more detailed expositions.  
\end{secnumber}

\begin{secnumber}
	Drinfeld's lemma also plays an essential role in V. Lafforgue's work \cite{lafforgue-v} on the automorphic-to-Galois direction of the Langlands correspondence for reductive groups over a global function field $F$. 
	Roughly speaking, V. Lafforgue showed that the space of cuspidal automorphic functions (with values in $\overline{\mathbb{Q}}_{\ell}$) of a reductive group $G$ over $F$ admits a decomposition indexed by certain $\ell$-adic Langlands parameters. 
	This decomposition is obtained by investigating the $\ell$-adic cohomology of certain moduli stacks of shtukas for $G$. 
	Moreover, he conjectured that this decomposition should be $\ell$-independent and indexed by certain motivic Langlands parameters. 
	
	In particular, we expect that there is a variant of V. Lafforgue's result in terms of $p$-adic Langlands parameters, corresponding to Abe's adaptation of the work of L. Lafforgue from $\ell$-adic to $p$-adic coefficients \cite{abe-companion}.
	In this adaptation, the $p$-adic analogues of of lisse $\ell$-adic sheaves are \textit{overconvergent $F$-isocrystals}. 
	Besides those objects, we may also consider the larger category of \textit{convergent $F$-isocrystals}, which admit no $\ell$-adic analogue but play an important role in the $p$-adic setup.
	
	Recently, Drinfeld proposed an unconditional definition of motivic Langlands parameters \cite{Dri18}. 
	Inspired by this work, we consider $p$-adic Langlands parameters as homomorphisms of the Tannakian group of the category of overconvergent $F$-isocrystals over a curve to the Langlands dual group of $G$. 
\end{secnumber}	

\begin{secnumber}
	From this perspective, we prove a Tannakian form of Drinfeld's lemma for overconvergent/convergent $F$-isocrystals, which aims to establish the aforementioned result for $p$-adic Langlands parameters. 
	Note that another ingredient of V. Lafforgue's approach, the \emph{geometric Satake equivalence}, was established for $F$-isocrystals by the second author and X. Zhu \cite{XZ}. 

	Keep the notations of \S~\ref{sss:intro setting}. 
	Given an overconvergent (resp. convergent) isocrystal $\mathscr{E}$ on $X$, a \emph{partial Frobenius structure} on $\mathscr{E}$ consists of two isomorphisms $\varphi_i\colon F_i^*(\mathscr{E})\xrightarrow{\sim} \mathscr{E}$ such that $\varphi_1\circ F_1^*(\varphi_2)= \varphi_2\circ F_2^*(\varphi_1)$; this composition in particular provides $
\mathscr{E}$ with the structure of an $F$-isocrystal.
	The category $\PhiIsocd(X)$ (resp. $\PhiIsoc(X)$) of overconvergent (resp. convergent) isocrystals with a partial Frobenius structure is a Tannakian category and we denote by $\pi_1^{\PhiIsocd}(X)$ (resp. $\pi_1^{\PhiIsoc}(X)$) the associated Tannakian group (with respect to a fiber functor). 
\end{secnumber}

\begin{theorem}[\ref{t:main thm}, \ref{t:Drinfeld lemma origin}, \ref{t:Drinfeld lemma conv}] \label{t:thm intro}
	\textnormal{(i)} 
	Suppose each $X_i$ is geometrically connected (resp. smooth and geometrically connected) over $k$.
	The pullback functors of projections $p_i\colon X\to X_i$ induce a canonical isomorphism of Tannakian groups:
	\begin{equation} \label{eq:iso intro}
	p_1^{\circ}\times p_2^{\circ}\colon 
	\pi_1^{\PhiIsocd}(X)\xrightarrow{\sim} \pi_1^{\FIsocd}(X_1)\times \pi_1^{\FIsocd}(X_2) \quad
	\textnormal{(resp. } \pi_1^{\PhiIsoc}(X)\xrightarrow{\sim} \pi_1^{\FIsoc}(X_1)\times \pi_1^{\FIsoc}(X_2)).
\end{equation}
	where $\pi_1^{\FIsocd}$ (resp. $\pi_1^{\FIsoc}$) denotes the Tannakian group of the category of overconvergent (resp. convergent) $F$-isocrystals over $\bQQ_p$, and similarly with $F$ replaced by $\Phi$. 

	\textnormal{(ii)}
	Suppose each $X_i$ is smooth and geometrically connected over $k$. 
	By taking connected components in the above isomorphism, we recover the isomorphism \eqref{eq:origin Drinfeld lemma}. 
\end{theorem}

In the overconvergent case, we also establish a similar isomorphism for overconvergent isocrystals which can be equipped with a Frobenius structure but do not carry a specified one. This roughly corresponds to passing from arithmetic to geometric fundamental groups in the $\ell$-adic setting. 

\begin{secnumber}
	In very recent work \cite{kedlaya-mono-rel}, the first author has established a relative version of the $p$-adic local monodromy theorem for differential modules with a Frobenius structure over an annulus \cite{And02II,Meb02,kedlaya-mono}. 
	An application is the local monodromy theorem for modules with an integrable connection and a partial Frobenius structure over polyannuli \cite[Theorem 3.3.6]{kedlaya-mono-rel}. 
	We formulate these results in terms of a Tannakian form of local Drinfeld's lemma (Theorem \ref{t:local Drinfeld lemma}) and discuss some related constructions. 

	We remark that variants of local Drinfeld's lemma for $\ell$-adic sheaves  are key ingredients of the local Langlands correspondence of Genestier--V. Lafforgue \cite{GL} and of Fargues--Scholze \cite{FS}. 

	Inspired by Berger's thesis \cite{Berger}, we expect to deduce a ``de Rham implies potentially semistable'' result for $p$-adic representations of powers of Galois groups from local Drinfeld's lemma. 
	There are some related results in this direction: (i) the overconvergence of multivariate $(\varphi,\Gamma)$-modules has been proved by the first author, A. Carter, and G. Z\'abr\'adi \cite{CKZ21}; (ii) multivariable de Rham representations and the associated $p$-adic differential equations are studied by Brinon, Chiarellotto, and Mazzari \cite{BCM21}. 
\end{secnumber}

\begin{secnumber}
        We now describe the structure of the paper.
	
  The proof of Theorem~\ref{t:thm intro} in the overconvergent case is given in \S\ref{sec-overconvergent}.	A key ingredient is Proposition \ref{P:pushforward},
   which says that the pushforward functor (for arithmetic $\mathscr{D}$-modules) of the projection $p_i\colon X\to X_i$ ($i=1,2$) sends overconvergent isocrystals on $X$, which can be equipped with a partial Frobenius structure, to overconvergent isocrystals on $X_i$. 
	Combined with a criterion of Esnault--Hai--Sun \cite{EPS} on exact sequences of Tannakian groups, we conclude Theorem \ref{t:thm intro}(i) in the overconvergent case. 

	The proof in the convergent case is contained in section \ref{s:conv case} and follows a similar line as in \cite{kedlaya-dri} (see below). 
	We study unit-root and diagonally unit-root convergent isocrystals with a partial Frobenius structure (Propositions \ref{P:Crew diagonal unitroot}, \ref{P:Crew partial Frob}), and the diagonal (resp. partial) Frobenius slope filtrations (Theorems \ref{t:slope filtration}, \ref{T:partial slope filtration}).
	These tools allow us to define a pushforward functor along the projection $p_i\colon X\to X_i$ from $\PhiIsoc(X)$ to $\FIsoc(X_i)$. 
	Then we deduce Theorem \ref{t:thm intro}(i) in the convergent case by a similar argument as in the overconvergent case.

	We upgrade the isomorphism \eqref{eq:iso intro} in an $\ell$-independent form à la Drinfeld \cite{Dri18} in section \ref{s:motivic} (see Proposition \ref{P:l-independence} and \eqref{eq:Drinfeld lemma mot}). 

	The last section is devoted to the local Drinfeld's lemma. 

\end{secnumber}

\begin{secnumber}
We now compare more carefully the results of this paper with those of \cite{kedlaya-dri}.
In \cite{kedlaya-dri}, each $X_i$ is required to be smooth over some perfect field $k_i$, but it is not required that $k_i$ be either finite or independent of $i$. The more restricted situation considered here (in which $k_i = \FF_p$ for all $i$) is sufficient for the application to geometric Langlands; moreover, it is not immediately apparent how to reproduce the Tannakian formulation in the setting of \cite{kedlaya-dri}. 

Another important difference with \cite{kedlaya-dri} is in the overall structure of the arguments. Therein, the convergent case is treated first, using the isomorphism \eqref{eq:origin Drinfeld lemma}
 as input, and then the overconvergent case is deduced as a corollary; herein, we obtain the overconvergent case directly using cohomological methods, deduce \eqref{eq:origin Drinfeld lemma}
 as a corollary, and finally recover the convergent case.
\end{secnumber}

\textbf{Notations}: 
Let $k$ be a perfect field of characteristic $p>0$, $K$ a complete discrete valuation field of characteristic zero with residue field $k$, and $\mathscr{O}_K$ its ring of integers. 
Let $s$ be a positive integer and $q=p^s$. 
We assume moreover that the $s$-th Frobenius automorphism $k\xrightarrow{\sim} k$, $x\mapsto x^q$ lifts to an automorphism $\sigma\colon \mathscr{O}_K\xrightarrow{\sim} \mathscr{O}_K$. 

	
In this article, we mainly work with the case where $(k,\sigma)=(\mathbb{F}_q,\id_{\mathscr{O}_K})$ (and $K$ is therefore a finite extension of $\mathbb{Q}_p$). 
	In this case, we fix an algebraic closure $\bQp$ of $K$.

	In the following, a \textit{$k$-variety} $X$ means a separated scheme of finite type over $k$.
	We denote the $s$-th Frobenius morphism on $X$ by $F_X$. 
	
	When using the notation $\prod$ for a product of schemes, the product is taken over $k$.

\section{Drinfeld's lemma for overconvergent $F$-isocrystals} \label{sec-overconvergent}
\subsection{Generalities on Tannakian categories}
	In the following, we study some general constructions in the Tannakian formalism. 
	Let $E$ be a field of characteristic zero.
	Let $\Vect_E$ denote the category of finite-dimensional $E$-vector spaces. 
	For any neutral Tannakian category $\mathscr{C}$ over $E$ with respect to a fiber functor $\omega$, we use $\pi_1(\mathscr{C},\omega)$ to denote its Tannakian group (i.e., the group of natural automorphisms of $\omega$). 
	We omit $\omega$ from the notation if there is no risk of confusion. 

	\begin{secnumber}\label{setting}
		Let $\widetilde{\mathscr{C}}$ be a Tannakian category over $E$, neutralized by a fiber functor $\omega\colon \widetilde{\mathscr{C}}\to \Vect_F$, where $F$ is a finite extension of $E$. Suppose there exist $E$-linear tensor equivalences $\tau_i\colon \widetilde{\mathscr{C}}\to \widetilde{\mathscr{C}}$ and isomorphisms of tensor functors $\eta_i\colon \omega\circ \tau_i\xrightarrow{\sim} \omega$ for $i=1,2,\dots,n$. 
	Moreover, we assume that there exist natural isomorphisms $\sigma_{ij}\colon \tau_i\circ \tau_j \simeq \tau_{j}\circ \tau_i$ for $i,j$ such that the following diagram commutes:
	\begin{equation} \label{sigma ij}
		\xymatrix{
		\omega\circ \tau_i\circ \tau_j \ar[r]^{\eta_i\circ \id} \ar[d]_{\omega(\sigma_{ij})} & \omega\circ \tau_j \ar[r]^{\eta_j} & \omega \\
		\omega \circ \tau_j\circ \tau_i \ar[r]^{\eta_j\circ \id} & \omega\circ \tau_i \ar[ru]_{\eta_i}
		}
	\end{equation}
	Since $\omega$ is faithful, such an isomorphism $\sigma_{ij}$ is unique.
	In the following, for every $i=1,\dots,n$, we fix a quasi-inverse $\tau_i^{-1}$ of $\tau_i$. For $m\in \mathbb{Z}$, we set $\tau_i^m$ to be the $|m|$-th composition of $\tau_i$ (or $\tau_i^{-1}$ if $m<0$) and define $\eta_i^m\colon \omega\circ \tau_i^m \xrightarrow{\sim} \omega$ by the composition of $\eta_i$ (or the inverse of $\eta_i$ if $m<0$). 

	We define a category $\mathscr{C}_0$ as follows: an object $(\mathscr{E},\varphi_1,\dots,\varphi_{n})$ consists of an object $\mathscr{E}$ of $\widetilde{\mathscr{C}}$ and isomorphisms $\varphi_i\colon \tau_i(\mathscr{E})\xrightarrow{\sim} \mathscr{E}$ such that $\varphi_j\circ \tau_j(\varphi_i)=\varphi_i\circ \tau_i(\varphi_j)$ via $\sigma_{ij}$. Morphisms of $\mathscr{C}_0$ are morphisms of $\widetilde{\mathscr{C}}$ compatible with the $\varphi_i$. 
	We have a canonical functor 
\begin{displaymath}
	\mathscr{C}_0\to \widetilde{\mathscr{C}}, \quad (\mathscr{E},\varphi_i)\mapsto \mathscr{E}.
\end{displaymath} 

By \cite[\S~2.5]{Del90} and a similar argument to \cite[Proposition 3.4.5]{XZ}, we can show that $\mathscr{C}_0$ is a Tannakian category over $E$ neutralized by $\omega$ over $\Vect_F$. 
	
	We say an object of $\mathscr{C}_0$ is \textit{constant} if its image in $\widetilde{\mathscr{C}}$ is isomorphic to a finite direct sum of the unit object.

	Recall \cite[\S~2]{Ber08} that a Tannakian subcategory of $\widetilde{\mathscr{C}}$ is a strictly full abelian subcategory closed under $\otimes$, duals, and subobjects (and thus quotients). 
	The constant objects of $\mathscr{C}_0$ form a Tannakian subcategory $\mathscr{C}_{0,\cst}$ of $\mathscr{C}_0$ (\textit{loc.cit}).

	Let $\mathscr{C}$ be the smallest Tannakian subcategory of $\widetilde{\mathscr{C}}$ containing the essential image of $\mathscr{C}_0$ (i.e. generated by the subquotients of the essential image of $\mathscr{C}_0$). 

\end{secnumber}

\begin{secnumber}
	In the following, we assume that $\widetilde{\mathscr{C}}$ is neutral over $E$ by $\omega\colon \widetilde{\mathscr{C}}\to \Vect_E$. 
	Then so are $\mathscr{C},\mathscr{C}_0$ and $\mathscr{C}_{0,\cst}$. 
	We prove the following results by a similar approach of \cite[Appendix]{DZ20}, where D'Addezio treated the case $n=1$. A similar discussion appeared in \cite[Appendix 2]{HNY} and \cite[\S~3.4]{XZ}. 
\end{secnumber}

\begin{prop} \label{ess of pione} 
	\textnormal{(i)} 
	The canonical functors $\mathscr{C}_{0,\cst}\to \mathscr{C}_0\to \mathscr{C}$ induce a short exact sequence 
	\begin{displaymath}
		1\to \pi_1(\mathscr{C}) \to \pi_1(\mathscr{C}_0)\to \pi_1(\mathscr{C}_{0,\cst})\to 1.
	\end{displaymath}

	\textnormal{(ii)} The category $\mathscr{C}_{0,\cst}$ is equivalent to the category of representations of $\mathbb{Z}^n$ over $E$ and the Tannakian group $\pi_1(\mathscr{C}_{0,\cst})$ is isomorphic to the pro-algebraic completion of $\mathbb{Z}^n$. 
\end{prop}
	
In view of \cite[Proposition 2.21]{DM} and the definition of $\mathscr{C}$, the morphism $\pi_1(\mathscr{C}) \to \pi_1(\mathscr{C}_0)$ (resp. $\pi_1(\mathscr{C}_0)\to \pi_1(\mathscr{C}_{0,\cst})$) is a closed immersion (resp. faithfully flat). 
	We use the following criterion to prove the exactness.

	\begin{theorem}[\cite{DM} Proposition 2.21, \cite{EPS} Theorem A.1] \label{exact criterion}
		We consider a sequence of affine group schemes over $E$:
		\begin{equation} \label{eq:exact criterion group sequence}
		L\xrightarrow{q} G\xrightarrow{p} A 
		\end{equation}
		 and the associated functors:
		\begin{equation}
		\Rep_E(A)\xrightarrow{p^*}\Rep_E(G)\xrightarrow{q^*} 
		\Rep_E(L).
		\end{equation}

		\textnormal{(i)} The map $p\colon G\to A$ is faithfully flat if and only if $p^*(\Rep_E(A))$ is a full subcategory of $\Rep_E(G)$, closed under taking subquotients. 

		\textnormal{(ii)} The map $q\colon L\to G$ is a closed immersion if and only if any object of $\Rep_E(L)$ is a subquotient of an object $q^*(V)$ for some $V\in \Rep_E(G)$. 

		\textnormal{(iii)} Assume that $q$ is a closed immersion and $p$ is faithfully flat. Then the sequence \eqref{eq:exact criterion group sequence} is exact if and only if the following conditions are satisfied.

\begin{itemize}
		\item[(a)] For an object $V\in \Rep_E(G)$, $q^*(V)$ in $\Rep_E(L)$ is trivial if and only if $V\simeq p^*U$ for some object $U$ in $\Rep_E(A)$. 

		\item[(b)] Let $W_0$ be the maximal trivial subobject of $q^*(V)$ in $\Rep_E(L)$. Then there exists a subobject $V_0$ of $V$ in $\Rep_E(G)$ such that $q^*(V_0)\simeq W_0$. 

		\item[(c)] Any object of $\Rep_E(L)$ is a subobject of an object in the essential image of $q^*$.
\end{itemize}

	\end{theorem}

	For any $E$-algebra $R$ and $1\le i \le n$, the above structure defines a homomorphism:
	\begin{equation} \label{def Z action}
		u_i\colon \mathbb{Z} \to \Aut(\pi_1(\mathscr{C})(R)),\quad k \mapsto (h\colon \omega\to \omega ~\mapsto~  \omega \xrightarrow{\eta_i^{-k}} \omega\circ \tau_i^k \xrightarrow{ h\circ \id}  \omega\circ \tau_i^k \xrightarrow{\eta_i^k} \omega). 
	\end{equation}
	In view of \eqref{sigma ij}, the images of $u_i,u_j$ commute with each other. 
	We thus obtain an action of $\mathbb{Z}^n$ on $\pi_1(\mathscr{C})$ and this allows us to define a group scheme $\pi_1(\mathscr{C})\rtimes \mathbb{Z}^n$ over $E$, which we denote by $W(\mathscr{C}_0)$.

\begin{lemma} \label{lemma Weil gp}
	\textnormal{(i)} There exist a canonical equivalence of categories $\mathscr{C}_0\xrightarrow{\sim} \Rep_{E}(W(\mathscr{C}_0))$ and a canonical morphism of group schemes $\iota\colon W(\mathscr{C}_0)\to \pi_1(\mathscr{C}_0)$ such that the following diagram is $2$-commutative:
	\begin{displaymath}
		\xymatrix{
		& \mathscr{C}_0 \ar[ld]^{\sim} \ar[rd]^{\sim} & \\
		\Rep_E(\pi_1(\mathscr{C}_0)) \ar[rr]^{\iota^*} & & \Rep_E(W(\mathscr{C}_0))
		}
	\end{displaymath}
	Moreover, the image of $\iota$ is Zariski-dense in $\pi_1(\mathscr{C}_0)$. 
	
	\textnormal{(ii)} The subgroup $\pi_1(\mathscr{C})$ of $\pi_1(\mathscr{C}_0)$ is normal. In particular, every object of $\mathscr{C}$ is a subobject of an object in the essential image of $\mathscr{C}_0$. 
\end{lemma}

\begin{proof}
	(i) We construct a functor
	\begin{displaymath}
		\mathscr{C}_0\to \Rep_E(W(\mathscr{C}_0)).
	\end{displaymath}
	Given an object $(\mathscr{E},\varphi_i)$ of $\mathscr{C}_0$, we construct a representation $\rho$ of $W(\mathscr{C}_0)(E)$ on $\omega(\mathscr{E})$. 
	For any element $(g,m_1,\dots,m_n)\in \pi_1(\mathscr{C})(E)\rtimes \mathbb{Z}^n$, we define $\rho(g,m_i)$ as the composition:
	\begin{displaymath}
		\omega (\mathscr{E})\xrightarrow{\prod \eta_{i}^{-m_i}} 
		\omega(\prod\tau_{i}^{m_i}(\mathscr{E})) \xrightarrow{g} 
		\omega(\prod\tau_{i}^{m_i}(\mathscr{E}))\xrightarrow{\prod \eta_{i}^{m_i}} 
		\omega(\mathscr{E}).
	\end{displaymath}
	In view the definition of $u_i$ \eqref{def Z action}, one checks that the above formula defines a representation. Then we obtain the above functor and we can check that it is an equivalence. 

	By the Tannakian reconstruction theorem \cite[Theorem 2.11]{DM},
	$\pi_1(\mathscr{C}_0)$ is the pro-algebraic completion of $W(\mathscr{C}_0)$ and the image of $\iota$ is therefore Zariski-dense. 

	(ii) The second assertion follows from the first one by condition \ref{exact criterion}(iii.c). 
	Let $H$ denote the normalizer of $\pi_1(\mathscr{C})$ in $\pi_1(\mathscr{C}_0)$. 
	As $\pi_1(\mathscr{C})$ is normal in $W(\mathscr{C}_0)$, the image of $\iota\colon W(\mathscr{C}_0)\to \pi_1(\mathscr{C}_0)$ is contained in $H$. 
	As $\iota$ has Zariski-dense image by (i), this implies that $H=\pi_1(\mathscr{C}_0)$ and the assertion follows. 
\end{proof}

\textit{Proof of Proposition \ref{ess of pione}.}
(i) We know the exactness at the left and the exactness on the right follows from \cite[Proposition 2.21]{DM}. We use Theorem \ref{exact criterion} to verify the exactness. Condition (a) follows from the definition. 
Given an object $(V,\varphi_1,\dots,\varphi_n)$ of $\mathscr{C}_0$, the maximal trivial subobject $W_0$ of $V$ in $\mathscr{C}$ is preserved by the $\varphi_i$. Therefore condition (b) is verified. Condition (c) is proved in Lemma \ref{lemma Weil gp}(ii).

(ii) The assertion follows from Lemma \ref{lemma Weil gp} applied to the category $\Vect_E$.
\hfill \qed

\subsection{Tannakian form of Drinfeld's lemma for overconvergent $F$-isocrystals}
\label{ss:Tannakian-gps}
\begin{secnumber} \label{sss:setting}
	Let $X$ be a $k$-variety. 
	We denote by $\Isocd(X/K)$ (resp. $\Isoc(X/K)$) the category of overconvergent (resp. convergent) isocrystals over $X$ relative to $K$. 
	If we apply the construction of \ref{setting} to the Tannakian category $\widetilde{\mathscr{C}}=\Isocd(X/K)$ (resp. $\Isoc(X/K)$) over $k$ and the ($s$-th) Frobenius pull-back functor $\tau=F_X^*$, then $\mathscr{C}_0$ corresponds to the category $\FIsocd(X/K)$ (resp. $\FIsoc(X/K)$) of overconvergent (resp. convergent) $F$-isocrystals over $X/(K,\sigma)$.
	We set $\Isod(X/K):=\mathscr{C}$, the Tannakian full subcategory of $\Isocd(X/K)$ generated by subquotients of overconvergent isocrystals which admit a ($s$-th) Frobenius structure, considered in \cite{abe-companion}.  
\end{secnumber}

\begin{secnumber}
	In the following, we assume $(k,\sigma)=(\mathbb{F}_q,\id_{\mathscr{O}_K})$. 

	The above construction can be generalized as follows. 
	For $i=1,\ldots,n$, let $X_i$ be a $k$-variety and set $X:=X_1\times_k \cdots \times_k X_n$. 
	We define a morphism $F_i\colon X\to X$ by the ($s$-th) Frobenius morphism $F_{X_i}$ of $X_i$ on the component $X_i$ and the identity map on other components.
	The morphisms $F_i$ commute with each other and their composition is equal to $F_X$. 
	The morphisms $F_i$ induce tensor equivalences 
	\[
	F_i^*\colon \Isocd(X/K)\to \Isocd(X/K)\quad \textnormal{(resp. } \Isoc(X/K)\to \Isoc(X/K)),\quad i=1,\ldots,n,
	\]
	commuting with each other and their composition is equal to $F^*_X$. 

	Given an overconvergent (resp. convergent) isocrystal $\mathscr{E}$ over $X/K$, \textit{a partial ($s$-th) Frobenius structure on $\mathscr{E}$} consists of isomorphisms $\varphi_i\colon F_i^{*}(\mathscr{E})\xrightarrow{\sim} \mathscr{E}$ for $i=1,\dots, n$ such that $\varphi_i\circ F_i^{*}(\varphi_j)=\varphi_j\circ F_j^{*}(\varphi_i)$. 
	For $i=1,\dots,n$, $\varphi_i$ is called the \textit{$i$-th partial Frobenius structure} (on $\mathscr{E}$). 
	Note that the composition of the $\varphi_i$ (in any order) forms a Frobenius structure on $\mathscr{E}$. 

	If we apply the construction of \ref{setting} to $(\widetilde{\mathscr{C}}=\Isocd(X/K)$ (resp. $\Isoc(X/K)$); $F_1^{*},\ldots,F_n^{*})$, then $\mathscr{C}_0$ corresponds to the category of overconvergent (convergent) isocrystals over $X$ with a partial Frobenius structure, which we denote by $\PhiIsocd(X/K)$ (resp. $\PhiIsoc(X/K)$). 
	We set $\Isod(X,\Phi/K):=\mathscr{C}$, the full subcategory of $\Isocd(X/K)$, generated by subquotients of overconvergent isocrystals which admits a $s$-th partial Frobenius structure.

	We denote by $F_i\textnormal{-}\Isocd(X/K)$ (resp. $F_i\textnormal{-}\Isoc(X/K)$) the category of pairs $(\mathscr{E},\varphi_i)$ consisting of an overconvergent (resp. convergent) isocrystal with an $i$-th partial Frobenius structure. 

	Let $k'/k$ be a finite extension of degree $a$ and set $K':=\rW(k')\otimes_{\rW(k)}K$. 
    (We denote the Witt vector functor by $\rW$ to reduce confusion with the notation $W(\mathscr{C}_0)$ from Lemma \ref{lemma Weil gp}.)
	Set $X_{k'}:=X\otimes_k k'$. Then we have a canonical functor of extension of scalars:
	\begin{equation}\label{eq:ext sca}
	\PhiIsoc(X/K)\to \PhiIsoc(X_{k'}/K'),\quad (\mathscr{E},\varphi_i)\mapsto (\mathscr{E}\otimes_K K',\varphi_i^{a}\otimes_K K').
\end{equation}
\end{secnumber}

\begin{secnumber} \label{sss:Tannakian gps}
	The category $\Isod(X,\Phi/K)$ defined above is a Tannakian category over $K$ and may not be neutral. 
	For any algebraic extension $L$ of $K$ in $\bQp$, we set $\Isod(X,\Phi/L):=\Isod(X,\Phi/K)\otimes_K L$. 
	When $L=\bQp$, this category is a neutral Tannakian category over $\bQp$ (with respect to a fiber functor) and we omit $/\bQp$ from the notation. 
	We denote by $\pi_1^{\Isod}(X,\Phi)$ (resp. $\pi_1^{\PhiIsocd}(X)$, resp. $\pi_1^{\PhiIsoc}(X)$) the Tannakian group of $\Isod(X,\Phi)$ (resp. $\PhiIsocd(X)$, $\PhiIsoc(X)$) over $\bQp$.
	When $n=1$ (i.e. partial Frobenius structures reduce to a Frobenius structure), we omit $\Phi$ from the notation or replace it with $F$.
	
	The pullback functor $p_{i}^*$ induces a canonical tensor functor
	\begin{eqnarray} \label{functor p_1*} \\
	\nonumber
		p_i^{*}\colon \FIsocd(X_i) &\to& \PhiIsocd(X) \quad 
		\textnormal{(resp. } \Isod(X_i)\to \Isod(X,\Phi), \textnormal{resp. } \FIsoc(X_i)\to \PhiIsoc(X)) \\
		\quad (\mathscr{E}_i,\varphi_i) &\mapsto& (p_i^{*}\mathscr{E}_i, \id,\ldots,p_i^{*}(\varphi_i),\ldots,\id). \nonumber 
	\end{eqnarray}
	By the Künneth formula, the functor $p_i^*\colon \Isod(X_i)\to \Isod(X,\Phi)$ is fully faithful.
	It induces a canonical $\bQp$-homomorphism
	\begin{displaymath}
		p_i^{\circ}\colon \pi_1^{\PhiIsocd}(X) \to \pi_1^{\FIsocd}(X_i) \quad \textnormal{(resp. } \pi_1^{\Isod}(X,\Phi)\to \pi_{1}^{\Isod}(X_i) , ~ \textnormal{resp. } \pi_1^{\PhiIsoc}(X) \to \pi_1^{\FIsoc}(X_i)).
	\end{displaymath}
\end{secnumber}

The Tannakian form of Drinfeld's lemma can be summarized as follows. Its proof will occupy most of the remainder of \S\ref{ss:Tannakian-gps}. 

\begin{theorem} \label{t:main thm}
	Assume each $X_i$ is a geometrically connected $k$-variety. 
	The following canonical homomorphisms are isomorphisms:
	\begin{equation} \label{Tannakian Drinfeld lemma}
		\prod_{i=1}^n p_i^{\circ}\colon
		\pi_1^{\Isod}(X,\Phi)\xrightarrow{\sim} \prod_{i=1}^n \pi_{1}^{\Isod}(X_i), \quad 
		\prod_{i=1}^n p_i^{\circ}\colon
		\pi_1^{\PhiIsocd}(X) \xrightarrow{\sim} \prod_{i=1}^n \pi_1^{\FIsocd}(X_i) . 
	\end{equation}
\end{theorem}

We first establish some basic properties of the category $\Isod(X,\Phi)$. 

	\begin{prop}
		\label{embedding of partial Frobenius}	
		\textnormal{(i)} An irreducible object of $\Isocd(X)$ belongs to $\Isod(X,\Phi)$ if and only if it can be equipped with an $s'$-th partial Frobenius structure for some $s|s'$. 

		\textnormal{(ii) [D'Addezio--Esnault \cite{DE20}]}
	The category $\Isod(X,\Phi)$ is closed under extension. Every object of $\Isod(X,\Phi)$ can be embedded into an object of $\PhiIsocd(X)$. 

		\textnormal{(iii)} The category $\Isod(X,\Phi)$ is equivalent to the thick full subcategory of $\Isocd(X)$ generated by those objects which can be equipped with an $s'$-th partial Frobenius structure for some $s|s'$. 
	\end{prop}
	\begin{proof}
	(i) Given an object $\mathscr{E}$ of $\PhiIsocd(X)$, the partial Frobenius pullbacks permute the isomorphism classes of the irreducible constituents of $\mathscr{E}$ in $\Isod(X)$. We thus conclude that an irreducible object $\Isod(X,\Phi)$ admits  $s'$-th partial Frobenius structure for some $s|s'$. 

	On the other hand, if an object $\mathscr{E}$ of $\Isocd(X)$ is irreducible equipped with an $s'(=st)$-th partial Frobenius structure for some $t\in \mathbb{N}$,
	then $\mathscr{E}$ is a subobject of \[
	\mathscr{F}:=\bigoplus_{j_1=0}^{t-1} \cdots \bigoplus_{j_n=0}^{t-1} F_1^{j_1*}\circ\cdots\circ F_{n}^{j_n*}(\mathscr{E}),
	\]
	 which admits an $s$-th Frobenius structure. 
	Therefore $\mathscr{E}$ belongs to $\Isod(X,\Phi)$. 

	(ii) When $n=1$, assertion (ii) is proved in \cite[Theorem 5.4]{DE20}. 
	The general case can be showed in a similar way as in \textit{loc.cit.}

	(iii) Assertion (iii) follows from assertions (i--ii). 
	\end{proof}

	\begin{secnumber}
	We will use the theory of holonomic (arithmetic) $\mathscr{D}$-modules and their six functors formalism developed in \cite{AC18,abe-companion}. 
	Let $L$ be an extension of $K$ in $\bQp$, $\mathfrak{T}:=\{k,\mathscr{O}_K,K,L\}$ the associated geometric base tuple and $\mathfrak{T}_F:=\{k,\mathscr{O}_K,K,L,s,\id_L\}$ the associated arithmetic base tuple \cite[1.4.10, 2.4.14]{abe-companion}. 

	Let $X$ be a $k$-variety. There exists an $L$-linear triangulated category $\rD(X/L)$ (resp. $\rD(X/L_F)$) relative to the geometric base tuple $\mathfrak{T}$ (resp. arithmetic base tuple $\mathfrak{T}_F$). 
	This category is denoted by $\rD^{\rb}_{\hol}(X/\mathfrak{T})$ or $\rD^{\rb}_{\hol}(X/L)$ (resp. $\rD^{\rb}_{\hol}(X/\mathfrak{T}_F)$ or $\rD^{\rb}_{\hol}(X/L_F)$) in \cite[1.1.1, 2.1.16]{abe-companion}. 
	For $\blacktriangle\in \{\emptyset,F\}$, there exists a \textit{holonomic t-structure} on $\rD(X/L_{\blacktriangle})$, whose heart is denoted by $\Hol(X/L_{\blacktriangle})$, called the \textit{category of holonomic modules}.
	We denote by $\hH^*$ the cohomological functor for holonomic t-structure. 
	
	The six functors formalism for $\rD(X/L)$ (resp. $\rD(X/L_F)$) has been established recently. We refer to \cite[\S~2.3]{abe-companion} for details and to \cite[1.1.3]{abe-companion} for a summary.
	\end{secnumber}

\begin{secnumber} \label{rem thick}
	Let $X$ be a smooth and quasiprojective geometrically connected $k$-variety of dimension $d_X$. 	
	The category $\Isod(X)$ (resp. $\FIsocd(X)$) is equivalent to the full subcategory $\Sm(X/\bQp)$ (resp. $\Sm(X/\bQpF)$) consisting of smooth objects of the category $\Hol(X/\bQp)[-d_X]$ (resp. $\Hol(X/\bQpF)[-d_X]$) of holonomic arithmetic $\mathscr{D}$-modules shifted by the dimension $-d_X$ \cite[1.1.3 (12)]{abe-companion}.  

	We briefly review the pushforward and pullback functors for a smooth morphism $f\colon X\to Y$ of relative dimension $d$ between quasiprojective geometrically connected $k$-varieties following \cite[1.2.8]{abe-companion}. 
	We have an adjoint pair $(f_+,f^+[2d])$
	\[
	f_+:\rD(X/L_{\blacktriangle}) \rightleftarrows \rD(Y/L_{\blacktriangle}): f^+[2d]. 
	\]
	The functor $f^+[d]$ is exact and induces a fully faithful functor \cite[Proposition 2.1.6]{XZ}:
	\begin{equation} \label{eq:fullyfaithful}
	f^+[d]\colon \Hol(Y/L_{\blacktriangle})\to \Hol(X/L_{\blacktriangle}).
\end{equation}
	If $d_X$ (resp. $d_Y$) denotes the dimension of $X$ (resp. $Y$), the functor 
	\begin{equation} \label{eq:pi pushforward Hol}
f_{*}:=\hH^{d}(f_{+}(-))[-d]\colon \Hol(X/L_{\blacktriangle})[-d_X]\to \Hol(Y/L_{\blacktriangle})[-d_Y]
\end{equation}
is a right adjoint of $f^*:=f^+$ and is left exact. 
Its right derived functor $\rR f_*$ is compatible with $f_+$. 

Moreover, there exists a canonical isomorphism (Poincar\'e duality \cite[1.5.13]{abe-companion}): 
	\[
	f^+(d)[2d]\xrightarrow{\sim} f^!.
	\]
\end{secnumber}

\begin{prop} \label{P:pushforward}
	Assume each $X_i$ is smooth and quasiprojective over $k$.

		\textnormal{(i)} The functor $p_{i,*}$ sends objects of $\Isod(X,\Phi)$ to $\Isod(X_i)$ (resp. $\PhiIsocd(X)$ to $\FIsocd(X_i)$). 

		\textnormal{(ii)} The functor $p_{i,*}\colon \Isod(X,\Phi)\to \Isod(X_i)$ is a right adjoint of $p_i^*$ \eqref{functor p_1*}.
		The adjoint morphism $\id\to p_{i,*}p_i^*$ is an isomorphism and $p_i^*p_{i,*}\to \id$ is injective.  
	\end{prop}

\begin{secnumber}
	Let $Y:=Y_1\times_k Y_2$ be a product of two geometrically connected $k$-varieties, $k'$ a perfect field over $k$, $\mathscr{O}_{K'}:=\rW(k')\otimes_{\rW(k)}\mathscr{O}_K$, equipped with a lift $\sigma'$ of the $s$-th Frobenius automorphism of $k'$ defined by that of $\rW(k')$ and $\id_{\mathscr{O}_K}$. We set $K':=\mathscr{O}_{K'}[\frac{1}{p}]$. 
	
	We denote by $F_1\textnormal{-}\Isocd(Y_{1,k'}/K')$ (resp. $F_1\textnormal{-}\Isoc(Y_{1,k'}/K')$) the category of pairs $(\mathscr{E},\phi)$ consisting of an object $\mathscr{E}$ of $\Isocd(Y_{1,k'}/K')$ (resp. $\Isoc(Y_{1,k'}/K')$) and an isomorphism $\phi\colon (F_{Y_1}\otimes \id_{k'})^*(\mathscr{E})\xrightarrow{\sim} \mathscr{E}$. 
	
	A point $i\colon \Spec(k')\to Y_2$ induces a natural functor: 
	\begin{eqnarray} \label{eq:partial Frob to FIsocd}
	\quad\quad	\iota\colon F_1\textnormal{-}\Isocd(Y/K) &\to& F_1\textnormal{-}\Isocd(Y_{1,k'}/K'), \quad 
	\textnormal{(resp. } 
	F_1\textnormal{-}\Isoc(Y/K)\to F_1\textnormal{-}\Isoc(Y_{1,k'}/K') ),\\
	(\mathscr{E},\varphi)&\mapsto& ((\id_{Y_1}\times i)^*(\mathscr{E}),(\id_{Y_1}\times i)^*\varphi). \nonumber 
\end{eqnarray}
When $k'$ is a finite extension of $k$ of degree $a$, we associate an $F$-isocrystal over $Y_{1,k'}/K'$ to an object of $F_1\textnormal{-}\Isocd(Y_{1,k'}/K')$ by composing its Frobenius structure $a$ times.  
	
	When $Y_2=\Spec(k)$, we regard the above functor as the functor of extension of scalars:
	\[
	\iota_{k'/k}\colon \FIsocd(Y/K) \to F_1\textnormal{-}\Isocd(Y_{1,k'}/K').
	\]
\end{secnumber}

\begin{secnumber}
	\textit{Proof of Proposition \ref{P:pushforward}}. 
	(i)
	It suffices to prove the assertion for objects of $\Isod(X,\Phi/K)$. We may assume $i=1$. We set $X':=\prod_{i=2}^n X_i$.
	
	(a) We first prove the assertion for an object $\mathscr{E}$ of $\Isod(X,\Phi/K)$ equipped with a partial Frobenius structure $(\varphi_1,\ldots,\varphi_n)$. 
	There exists an open subscheme $U_1$ of $X_1$ such that $p_{1,*}(\mathscr{E})|_{U_1}$ is smooth. In view of their fibers, the adjoint morphism of overconvergent isocrystals:
	\[
		p_1^*(p_{1,*}(\mathscr{E})|_{U_1})\to \mathscr{E}|_{U_1\times_k X'}
	\]
		is injective. 
		The first partial Frobenius $\varphi_1$ on $\mathscr{E}$ induces a Frobenius structure $\phi_1$ on $p_{1,*}(\mathscr{E})|_{U_1}$ 		and the above morphism is compatible with the first partial Frobenius structures $p_1^*(\phi_1)$ and $\varphi_1$. 
		Let $k'$ be a perfect closure of $k(X')$ and $K':=\rW(k')[\frac{1}{p}]$. 
		By the exactness of the functor $\iota$, defined in \eqref{eq:partial Frob to FIsocd}, we obtain an injection in $F_1\textnormal{-}\Isocd(U_{1,k'}/K')$:
		\[
		\iota_{k'/k}(p_{1,*}(\mathscr{E})|_{U_1}) \hookrightarrow \iota(\mathscr{E}|_{U_{1}\times_k X'}).
		\]
		By \cite[Proposition 5.3.1]{Ked07}, the left hand side extends to a subobject $\mathscr{F}'$ of $\iota(\mathscr{E})$ of $\Isocd(X_{1,k'}/K')$. 

		We claim that there exists an overconvergent isocrystal $\mathscr{F}$ of $\Isocd(X_1/K)$ extending $p_{1,*}(\mathscr{E})|_{U_1}$ such that $\iota_{k'/k}(\mathscr{F})\simeq \mathscr{F}'$. 
	By Kedlaya--Shiho's purity theorem \cite[Theorem 5.1]{kedlaya-isocrystals}, we may assume that the boundary $D:=X_1-U_1$ is a smooth divisor. 
	Since the local monodromy of $\mathscr{F}'$ around $D_{k'}$ is constant \cite[Theorem 5.2.1]{Ked07}, then so is the local monodromy of $p_{1,*}(\mathscr{E})|_{U_1}$ around $D$. Then the claim follows. 
		
	Moreover, the Frobenius structure on $p_{1,*}(\mathscr{E})|_{U_1}$ extends to $\mathscr{F}$. 
	Finally, by full faithfullness of pullback for holonomic $\mathscr{D}$-modules along $U_1\to X_1$ \cite[Proposition 2.1.6]{XZ}, we deduce that $\mathscr{F}$ is isomorphic to $p_{1,*}(\mathscr{E})$.
	Then the assertion in this case follows. 
	
		(b) An object $\mathscr{E}$ of $\Isod(X,\Phi)$ can be embedded into an object $\mathscr{F}$, which can be equipped with a partial Frobenius structure (Proposition \ref{embedding of partial Frobenius}(ii)).  
		As $p_{1,*}$ is left exact, we have an injection to a smooth object $\mathscr{N}$ of $\Hol(X_1/\bQpF)[-d_{X_1}]$: 
		\begin{displaymath}
			\mathscr{M}:=p_{1,*}(\mathscr{E})\hookrightarrow \mathscr{N}:=p_{1,*}(\mathscr{F}).
		\end{displaymath}

		It remains to show that $\mathscr{M}$ is also smooth. 
		We first consider the case where $X_1$ is a curve. 
		Let $U_1$ be an open subset of $X_1$ on which $\mathscr{M}$ is smooth.
		By applying \cite[Corollary 2.3.4]{XZ} to $\mathscr{M}|_{U_1}$, the exactness of pullback on smooth modules and of the nearby cycle functor, we deduce the smoothness of $\mathscr{M}$ from that of $\mathscr{N}$.

In general, there is a dense open subscheme $j\colon U_1\to X_1$ and a smooth object $\mathscr{L}$ on $U_1$ such that $\mathscr{M}|_{U_1}\simeq \mathscr{L}$. 
Let $c\colon C\to X_1$ be a morphism from a smooth curve to $X$ such that $c(C)\cap U_1$ is non-empty and $p_c\colon C\times_k X'\to C$ the projection. 
By the above argument, $p_{c,*}( (c\times\id)^*\mathscr{E})$ is smooth and the pullback $c^*\mathscr{L}$ extends to an overconvergent isocrystal on $C$. 
By Shiho's cut-by-curves theorem \cite{Shiho11} and Kedlaya--Shiho's purity theorem, $\mathscr{L}$ can be extended to an overconvergent isocrystal over $X$. 
By full faithfulness of pullback along $U_1\to X_1$, we deduce that $\mathscr{M}$ is isomorphic to the extension of $\mathscr{L}$ to $X_1$. This finishes the proof. 

	(ii) The isomorphism $\id\to p_{i,*}p_i^*$ follows from the Künneth formula \cite[Proposition 1.1.7]{abe-companion}. 
	In view of fibers at closed points, the injectivity of $p_i^*p_{i,*}\to \id$ follows.  \hfill \qed
\end{secnumber}

\begin{secnumber} 
\label{pf main thm}
We now turn to the proof of Theorem \ref{t:main thm}. 
We first prove Theorem \ref{t:main thm} for $\pi_1^{\Isod}$ under the additional hypothesis that each $X_i$ is \textit{smooth and quasiprojective} over $k$. 

\textit{Proof.} As we work with arithmetic $\mathscr{D}$-modules with coefficients in $\bQp$, we may enlarge the base field $k$ as in \cite[1.4.11]{abe-companion}. 
		Therefore we may assume there exists a $k$-point $x$ of $X_1$. 
		We consider the following diagram
	\begin{displaymath}
		\xymatrix{
			X'\ar[r]^-{u=(x,\id)} \ar[d]_{g} & X=X_1\times_k X' \ar[d]_{p_1} \\
			\Spec(k)\ar[r]^-{x} & X_1
		}
	\end{displaymath}	

	We denote by $\PhipIsocd(X')$ the category of overconvergent isocrystals over $X'$ with a partial Frobenius structure and $\Isod(X',\Phi')$ the Tannakian subcategory of $\Isocd(X')$ generated by those object which can be equipped with some partial Frobenius structure.  
	The morphism $u$ induces a tensor functor 
	\begin{equation} \label{functor u*}
		u^{*}\colon \Isod(X,\Phi)\to \Isod(X',\Phi')\quad \textnormal{(resp. } \PhiIsocd(X) \to \PhipIsocd(X')) \quad \mathscr{E}\mapsto u^{*}(\mathscr{E})
	\end{equation}
	and a homomorphism 
	$u^{\circ}\colon \pi_1^{\Isod}(X',\Phi')\to \pi_1^{\Isod}(X,\Phi)$.
	Consider the following commutative diagram:
	\begin{equation} \label{eq:compare tannakian exact sequences}
		\xymatrix{
		1\ar[r] & \pi_1^{\Isod}(X',\Phi') \ar[r]^{u^{\circ}} \ar[d]_{\prod_{i=2}^n p_i^{\circ}} & \pi_1^{\Isod}(X,\Phi) \ar[r]^{p_1^{\circ}} \ar[d]_{\prod_{i=1}^n p_i^{\circ}} & \pi_1^{\Isod}(X_1) \ar[r] \ar@{=}[d] & 1 \\
		1\ar[r] & \prod_{i=2}^n \pi_1^{\Isod}(X_i)\ar[r] & \prod_{i=1}^n \pi_1^{\Isod}(X_i) \ar[r] & \pi_1^{\Isod}(X_1) \ar[r] & 1
		}
	\end{equation}
	Here the upper sequence is defined as above and the lower sequence is exact. 
	By induction, it suffices to show the exactness of the upper sequence. 
	In the following, we will do it by checking the conditions of Theorem \ref{exact criterion}
	for the functors:
	\begin{displaymath}
		\Isod(X_1)\xrightarrow{p_1^{*}} \Isod(X,\Phi) \xrightarrow{u^{*}} \Isod(X',\Phi').
	\end{displaymath}
	where $p_1^{*}$ is defined as in \eqref{functor p_1*}.
	We use the functor $p_{i,*}$ defined in Proposition~\ref{P:pushforward}(i).

	(i)
	Let $\mathscr{E}$ be an object of $\Isod(X_1)$ and $\mathscr{F}$ a subobject of $p_1^*(\mathscr{E})$. We need to show that $\mathscr{F}$ is the image under $p_1^*$ of a subobject of $\mathscr{E}$. If $\mathscr{G}$ denotes the quotient $p_1^*(\mathscr{E})/\mathscr{F}$, by applying $p_1^*p_{1,*}$ we obtain a commutative diagram:
	\begin{displaymath}
		\xymatrix{
		0\ar[r] & p_1^*p_{1,*}\mathscr{F}\ar[r] \ar[d]_{\alpha}& p_1^*\mathscr{E} \ar[r]^{\delta} \ar[d]_{\beta} & p_1^*p_{1,*}\mathscr{G} \ar[d]_{\gamma} &\\
		0\ar[r] & \mathscr{F}\ar[r] & p_1^*\mathscr{E} \ar[r] & \mathscr{G}\ar[r] &0 
		}
	\end{displaymath}
	Here $\beta$ is the identity, while $\alpha,\gamma$ are injective by Proposition~\ref{P:pushforward}(ii). Then $\delta$ is surjective. We deduce that $\alpha$ is an isomorphism. 
	Then we conclude that $\pi_1^{\Isod}(X,\Phi) \to \pi_1^{\Isod}(X_1)$ is faithfully flat by \cite[Proposition 2.21]{DM}. 

	(ii) and (iii.c) Since $u$ is a section of $p'\colon X=X_1\times_k X'\to X'$, we deduce condition (iii.c) of Theorem \ref{exact criterion} for $u^{\circ}\colon \pi_1^{\Isod}(X',\Phi')\to \pi_1^{\Isod}(X,\Phi)$. Then $u^{\circ}$ is a closed immersion by \cite[Proposition 2.21]{DM}. 

	(iii.b) 
	Given an object $\mathscr{E}$ of $\Isod(X,\Phi)$, the maximal trivial subobject of $u^{*}\mathscr{E}$ in the category $\Isod(X',\Phi')$ is $g^{*}g_{*}u^{*}\mathscr{E}$. 
	Since $u^+(\mathscr{E})[2d]\simeq u^!(\mathscr{E})$, we have $g^{*}g_{*}u^{*}\mathscr{E}\simeq u^* p_1^*p_{1,*}\mathscr{E}$ by smooth base change. 
	Then condition (b) follows from the fact that $p_1^*p_{1,*}\to \id$ is injective by Proposition~\ref{P:pushforward}(ii).

	(iii.a) If an object $\mathscr{E}$ of $\Isod(X,\Phi)$ comes from the essential image of $p_1^{*}$, then $u^*(\mathscr{E})$ is trivial. 
	Conversely, in view of (iii.b), if $u^*(\mathscr{E})$ is trivial, then $\mathscr{E}\simeq p_1^*p_{1,*}(\mathscr{E})$. This verifies condition (a). \hfill\qed
\end{secnumber}

We now upgrade the previous argument to eliminate the smooth quasiprojective hypothesis.

\begin{lemma} \label{lem:descent}
The categories $\Isocd(X), \PhiIsocd(X), \Isod(X, \Phi)$ admit descent with respect to any proper hypercoverings of the $X_i$.
\end{lemma}
\begin{proof}
Descent for $\Isocd(X)$ is established in \cite[Theorem~5.1]{lazda}.
This then implies descent for $\PhiIsocd(X)$; we deduce descent for $\Isod(X, \Phi)$ using Proposition~\ref{embedding of partial Frobenius}.
\end{proof}

\begin{secnumber}
\textit{Proof of Theorem \ref{t:main thm} for $\pi_1^{\Isod}$ with general $X_i$}. 	
By de Jong's alterations theorem \cite{dejong-alterations}, there exist smooth, connected and quasiprojective $k$-varieties $Y_i$ and proper surjective, generic \'etale maps $\pi_i\colon Y_i\to X_i$ for $i=1,2,\ldots,n$. 
We use the hypercovering produced by the $\pi_i$ to show that the 
upper sequence in \eqref{eq:compare tannakian exact sequences} is exact, by again checking the conditions of Theorem \ref{exact criterion}.

We set $Y:=\prod_{i=1}^n Y_i$ and $\pi\colon Y\to X$ to be the product $\prod_{i=1}^n \pi_i$. 
By induction, it suffices to treat the case where $Y_i=X_i$ for $i=2,\dots,n$. Consider the following diagram:
\[
\xymatrix{
	1\ar[r] & \pi_1^{\Isod}(X') \ar[r]^{u^{\circ}} \ar@{=}[d] & \pi_1^{\Isod}(Y,\Phi) 
	\ar[r]^{p_1^{\circ}} \ar[d]_{\pi^{\circ}} & \pi_1^{\Isod}(Y_1) \ar[r] \ar[d]_{\pi_1^{\circ}} & 1 \\
	1\ar[r] & \pi_1^{\Isod}(X') \ar[r]^{u^{\circ}} & \pi_1^{\Isod}(X,\Phi) \ar[r]^{p_1^{\circ}} & \pi_1^{\Isod}(X_1) \ar[r] & 1 
}
\]
By induction hypotheses, we may assume the exactness of the second line.  
We will deduce the exactness of the first line from the second one. 
The above diagram corresponds to the functor:
\[
\xymatrix{
\Isod(X_1) \ar[r]^{p_1^*} \ar[d]_{\pi_1^*} & \Isod(X,\Phi) \ar[r]^{u^*} \ar[d]_{\pi^*} & \Isod(X') \ar@{=}[d] \\ 
	\Isod(Y_1) \ar[r]^{p_1^*} \ar[d] & \Isod(Y,\Phi) \ar[r]^{u^*} \ar[d] & \Isod(X') \\
	\Isod(Y_1^{(2)}) \ar[r] & \Isod(Y^{(2)}, \Phi) & 
}
\]
Here $Y_1^{(2)}:=Y_1\times_{X_1}Y_1$ and $Y^{(2)}:=Y\times_X Y$. 
We check the conditions of \ref{exact criterion} for the first line. 

(i)  
Let $U$ be an object of $\Isod(X_1)$ and $V_0$ a subobject of $p_1^*(U)$. 
By exactness of the second line, we deduce that there exists a subobject $U_{0,Y}$ of $\pi_1^*(U)$ over $Y_1$ such that $p_1^*(U_{0,Y})\simeq \pi^*(V_0)$ as 
subobjects of $\pi^*(p_1^*(U))$.
As the functor $p_1^*\colon \Isod(Y_1^{(2)})\to \Isod(Y^{(2)},\Phi)$ is fully faithful, the descent data on $\pi^*(V_0)$ gives rise to a descent data on $U_0$. 
In this way, we obtain a subobject $U_{0}$ of $U$ over $X_1$, sent to $V_0\to p_1^*(U)$ via $p_1^*$. 
Hence the map $p_1^\circ\colon \pi_1^{\Isod}(X,\Phi)\to \pi_1^{\Isod}(X_1)$ is faithfully flat.

Conditions (ii) and (iii.c) follow from the same argument as in the smooth and quasiprojective case. 

(iii.a) Let $V$ be an object of $\Isod(X,\Phi)$ such that $u^*(V)$ is trivial. Then there exists an object $U_Y$ over $Y_1$ such that $p_1^*(U_Y)\simeq \pi^*(V)$. 
The descent data on $\pi^*(V)$ induces a descent data on $U_Y$, which gives rise to an object $U$ on $X_1$, sent to $V$ via $p_1^*$. 

(iii.b) Let $V$ be an object of $\Isod(X,\Phi)$, $W_0$ be the maximal trivial subobject of $u^*(V)$. 
By exactness of the second line, there exists a subobject $V_0$ of $\pi^*(V)$ on $Y$, sent to $W_0$ via $u^*$. 
Since $u^*$ is faithful, we deduce that the descent data on $\pi^*(V)$ preserves $V_0$. 
This gives rise to a subobject of $V$, sent to the trivial object $W_0$.
\hfill \qed 
\end{secnumber}
\begin{secnumber}
	\textit{Proof of Theorem \ref{t:main thm} for $\pi_1^{\PhiIsocd}$}. 
	By \ref{ess of pione}(i), we have a commutative diagram: 
	\begin{displaymath}
		\xymatrix{
		1\ar[r] & \pi_1^{\Isod}(X,\Phi) \ar[r] \ar[d]_{\prod p_i^{\circ}} & \pi_1^{\PhiIsocd}(X) \ar[r] \ar[d]_{ \prod p_i^{\circ}} & \pi_1^{\PhiIsocd}(X)^{\cst} \ar[r] \ar[d] & 1 \\
		1\ar[r] & \prod_{i=1}^n \pi_1^{\Isod}(X_i) \ar[r] & \prod_{i=1}^n \pi_1^{\FIsocd}(X_i) \ar[r] & \prod_{i=1}^n \pi_1^{\FIsocd}(X_i)^{\cst} \ar[r] & 1
		}
	\end{displaymath}
	where the first and second lines are exact. 
	By \ref{ess of pione}(ii), the right vertical arrow identifies with the projection of the pro-algebraic completion of $\mathbb{Z}^n$ to the pro-algebraic completion of $\mathbb{Z}$ and is therefore an isomorphism. 
	The assertion follows from that the left vertical arrow is an isomorphism.  \hfill \qed
\end{secnumber}

\section{Drinfeld's lemma for convergent $\Phi$-isocrystals}
\label{s:conv case}
In this section, we assume $(k,\sigma)=(\mathbb{F}_q,\id_{\mathscr{O}_K})$, and $X_i$ denotes a $k$-variety for $i=1,\ldots,n$ and $X:=\prod_{i=1}^n X_i$. 
Except in \S~\ref{sss:partial Frob slope}--\ref{def:unit-root}, we assume each $X_i$ is smooth and geometrically connected over $k$. 

\subsection{Unit-root and diagonally unit-root convergent $\Phi$-isocrystals}
\begin{secnumber} \label{sss:partial Frob slope}
	We first define the \textit{partial Frobenius slopes} at a closed point of $X$. 

	Let $(\mathscr{E},\varphi_i)$ be an object of $\PhiIsoc(X/K)$, $x\in |X|$ a closed point, $K_x:=\rW(k_x)\otimes_{\rW(k)}K$ and $a:=[k_x:k]$. 
	We take the extension of scalars to $k_x$ \eqref{eq:ext sca} and then take its fiber $(\mathscr{E}_x,\varphi_1^a|_x,\ldots,\varphi_n^a|_x)$ at $x$, which is an object of $\PhiIsoc(\Spec(k_x)/K_x)$ with respect to the $sa$-th power of Frobenius, that is a $K_x$-vector space together with commuting linear automorphisms $\varphi_i^a|_x$. 
	This allows us to define the \textit{$i$-th partial Frobenius slopes of $(\mathscr{E},\varphi_i)$ at $x$} for $i=1,\ldots,n$ by Dieudonn\'e--Manin theory. 

	We have a surjection from points of $X$ to products of points of the $X_i$:
	\[
	\pi\colon |X|\twoheadrightarrow \prod_{i=1}^n |X_i|.
	\]

	Let $(x_i)_{i=1}^n\in \prod |X_i|$ be a tuple of closed points. To simplify notation, assume that there exists a finite extension $k'$ of $k$ such that $k_{x_i}\simeq k'$ (this can always be enforced by enlarging $k$). 
	We have a decomposition 
	\begin{equation} \label{eq:decomposition points}
		\prod_k x_i\simeq \bigsqcup_G \Spec(k'),
	\end{equation}
	indexed by $G:=\left( \prod \Gal(k'/k) \right)/\Gal(k'/k)$, for the diagonal action. 
	For an object $(\mathscr{E},\varphi_i)$ of $\PhiIsoc(\prod_k x_i/K)$, $\mathscr{E}$ corresponds to a direct sum of vector spaces over $K'$, indexed by $G$. 
	If $a=[k':k]$, then the $a$-th power of each partial Frobenius $\varphi_{i}^a$ preserves each component. 
\end{secnumber}

	\begin{lemma}
		\label{l:partial Frob}
		The partial Frobenius slopes at a point $x\in |X|$ depend only on its image $\pi(x)=(x_1,\ldots,x_n)\in \prod |X_i|$. 
	\end{lemma}
	\begin{proof}
		We keep the above notation. Since $G$ acts transitively on each component of \eqref{eq:decomposition points}, the action induces isomorphisms between the pullback of $(\mathscr{E}_x,\varphi_{i,x}^a)$ to each component and the claim follows. 
\end{proof}

\begin{defn} \label{def:unit-root}
	Let $(\mathscr{E},\varphi_1,\ldots,\varphi_n)$ be an object of $\PhiIsoc(X)$. 

	(i) We say the $i$-th partial Frobenius structure $\varphi_i\colon F_i^*(\mathscr{E})\to \mathscr{E}$ is \textit{unit-root}, if its slope at each closed point is zero.  
	
	(ii) We say $\mathscr{E}$ is \textit{unit-root}, if every partial Frobenius structure $\varphi_i$ is unit-root. 

	(iii) We say $\mathscr{E}$ is \textit{diagonally unit-root} if the associated convergent $F$-isocrystal is unit-root. 
\end{defn}

\begin{secnumber} \label{sss:crew equiv}
	In the following, we assume each $X_i$ is a smooth geometrically connected $k$-variety. Then so is $X$.  

	Let $x$ be a closed point of $X$ and $\overline{x}$ the associated geometric point. 
	Recall \cite{crew-f} that there exists a canonical equivalence between the category of continuous $K$-representations of $\pi_1^{\et}(X,\overline{x})$ and the category of unit-root convergent $F$-isocrystals over $X/K$:
	\begin{equation}\label{eq:crew}
		\Rep^{\cont}_{K}(\pi_1^{\et}(X,\overline{x}))\xrightarrow{\sim} \FIsoc^{\ur}(X/K).
	\end{equation}
	Via the equivalence between the left hand side and the category $\LS(X,K)$ of lisse $K$-sheaves over $X$, we have:
	\begin{equation} \label{eq:crew LS}
		\LS(X,K)\xrightarrow{\sim} \FIsoc^{\ur}(X/K).
	\end{equation}
	The above equivalence is compatible with  the following operations: 
	\begin{itemize}
	\item[(i)] The pullback functoriality along a morphism between smooth connected $k$-varieties. 
	\item[(ii)] Extensions of scalars. 
	\end{itemize}

	Recall that a lisse $K$-sheaf $\mathbb{L}$ over $X$ is equipped with the Frobenius correspondence \cite[XIV=XV]{SGAV}:
	\begin{equation} \label{eq:phiV}
		\phi_{\mathbb{L}}\colon F_X^*(\mathbb{L})\xrightarrow{\sim} \mathbb{L}. 
	\end{equation}
\end{secnumber}

\begin{lemma}
	Let $(\mathscr{E},\varphi)$ be a unit-root convergent $F$-isocrystal over $X$ and $\mathbb{L}$ the associated sheaf. 
	Then the Frobenius structure $\varphi\colon (F_X^*\mathscr{E},F_X^*\varphi)\xrightarrow{\sim} (\mathscr{E},\varphi)$ gives rise to the above isomorphism $\phi_{\mathbb{L}}$. 
\end{lemma}
\begin{proof}
	In view of the construction of \eqref{eq:crew}, we may reduce to the case where $\mathbb{L}$ is a lisse $\mathscr{O}_K/p^n \mathscr{O}_K$-sheaf for some $n$, and we may assume $X$ is affine and admits a smooth lifting $\XX_n$ to $\mathscr{O}_K/p^n\mathscr{O}_K$, equipped with a Frobenius lift. 
	After taking pullback along a finite \'etale morphism trivializing $\mathbb{L}$, we may moreover assume that $\mathbb{L}$ is a trivial $\mathscr{O}_K/p^n\mathscr{O}_K$-module. 
	In this case, the assertion is clear. 
\end{proof}
\begin{remark} \label{rem:Frob on local system}
	When $X=x=\Spec(k)$, we will view $\phi_{\mathbb{L}}$ as an endomorphism on $\mathbb{L}$ in such a way that the action on the geometric fiber $\mathbb{L}_{\overline{x}}$ coincides with the action of the geometric Frobenius $F_k\in \Gal(\overline{k}/k)$.
\end{remark}

\begin{secnumber} \label{sss:Crew diagonal unitroot}
	Let $i$ be an integer $\in [1,n]$. The $i$-th partial Frobenius $F_i$ of $X$ is a homeomorphism and induces an equivalence of \'etale topoi of $X$.  
	Let $(\mathscr{E},\varphi_1,\ldots,\varphi_n)$ be a diagonally unit-root convergent $\Phi$-isocrystal over $X$ and $\mathbb{L}$ the lisse $K$-sheaf over $X$ associated to $(\mathscr{E},\varphi)$. 
	The $i$-th partial Frobenius structure $\varphi_i$ induces an isomorphism of sheaves:
	\[
	\phi_i\colon F_i^*(\mathbb{L})\xrightarrow{\sim} \mathbb{L}.  
	\]
	The isomorphisms $\phi_i$ commute with each other in the following sense: for any $1\le i,j \le n$, the identifications
	\[
	F_i^* (F_j^* (\mathbb{L})) \cong (F_i \circ F_j)^* (\mathbb{L}) = (F_j \circ F_i)^* (\mathbb{L}) \cong F_j^*( F_i^* (\mathbb{L}))
	\]
	induce an equality $\phi_j\circ F_j^*(\phi_i)=\phi_i\circ F_i^*(\phi_j)$.
	The composition of $\phi_i$ coincides with the isomorphism $\phi_{\mathbb{L}}$ \eqref{eq:phiV}. 
	The above construction is clearly functorial. 
\end{secnumber}

\begin{prop} \label{P:Crew diagonal unitroot}
	The category of diagonally unit-root $\Phi$-isocrystals over $X$ is equivalent to the category of pairs $(\mathbb{L},\{\phi_i\}_{i=1}^n)$ consisting of a lisse $K$-sheaf $\mathbb{L}$ over $X$ together with isomorphisms $\phi_i\colon F_i^*(\mathbb{L})\xrightarrow{\sim} \mathbb{L}$ commuting to each other and whose composition is $\phi_{\mathbb{L}}$ \eqref{eq:phiV}. 
	The morphisms in the latter category are morphisms of $\LS(X,K)$ compatible with the $\phi_i$. 
\end{prop}

\begin{proof}
	We construct a quasi-inverse of the functor in \S~\ref{sss:Crew diagonal unitroot}. 
	Let $(\mathbb{L},\{\phi_i\}_{i=1}^n)$ be a collection of data as above and $(\mathscr{E},\varphi)$ the unit-root $F$-isocrystal over $X$ associated to $\mathbb{L}$. 
	By functoriality of \eqref{eq:crew LS}, $\phi_i$ induces an $i$-th partial Frobenius structure $\varphi_i$ on $\mathscr{E}$. 
	The commutativity of the $\varphi_i$, and the fact that the composition of the $\varphi_i$ equals $\varphi$, follow from the corresponding properties of the $\phi_i$. 
	This construction is clearly functorial and provide a quasi-inverse of \S~\ref{sss:Crew diagonal unitroot}. 
\end{proof}

\begin{cor} \label{c:constant partial Frob}
	Let $(\mathscr{E},\varphi_i)$ be a diagonally unit-root $\Phi$-isocrystal over $X$. 
	Then for $i=1,2,\ldots,n$, the $i$-th partial Frobenius slopes of $\mathscr{E}$ are constant on $|X|$. 	
\end{cor}
\begin{proof}
	Let $(\mathbb{L},\{\phi_i\}_{i=1}^n)$ be the associated data in Proposition~\ref{P:Crew diagonal unitroot}. 
	Then the $i$-th partial Frobenius slope of $\mathscr{E}$ at $x\in X(k)$ can be calculated by that of $\phi_{i,\overline{x}}$ on $\mathbb{L}_{\overline{x}}$ (with the convention of Remark~\ref{rem:Frob on local system}). 
	Since $\mathbb{L}$ admits a lisse $\mathscr{O}_K$-sheaf as an integral model, the slopes of $\phi_i$ at each fibers $\mathbb{L}_{\overline{x}}$ are constant as function on $x\in X(k)$.

	The corollary follows from applying the previous argument to extensions of scalars of $(\mathscr{E},\varphi_i)$ (\S~\ref{sss:crew equiv}).  
\end{proof}

\begin{theorem}
	\label{t:slope filtration}
	Let $(\mathscr{E},\varphi_i)$ be a convergent $\Phi$-isocrystal over $X$. 
	Suppose that the diagonal Newton polygon of $(\mathscr{E},\varphi_i)$ is constant as a function on $|X|$. 
	
	\textnormal{(i)} 
	There exists a filtration
	\begin{equation} \label{eq:d slope filtration}
0 = \mathscr{E}_0 \subset \cdots \subset \mathscr{E}_l = \mathscr{E}
\end{equation}
	of $\PhiIsoc(X)$ and an increasing sequence $\mu_1<\mu_2<\cdots<\mu_{\ell}$ of rational numbers such that for $j=1,\dots,\ell$, the diagonal Newton polygon of $\mathscr{E}_j/\mathscr{E}_{j-1}$ is constant with slope $\mu_j$. 
	Moreover, the filtration and sequence are both uniquely determined by this condition. (We call it the \emph{diagonal slope filtration} of $\mathscr{E}$.)
	
	\textnormal{(ii)} 
	For each partial Frobenius structure $\varphi_i$ of $\mathscr{E}$, its Newton polygon is also constant on $|X|$. 
\end{theorem}

\begin{proof}
	(i) There exists a slope filtration \eqref{eq:d slope filtration} as convergent $F$-isocrystals \cite[Corollary~4.2]{kedlaya-isocrystals}. 
	It suffices to show that each partial Frobenius $\varphi_i$ preserves this filtration, that is the composition $F_i^*(\mathscr{E}_j)\xrightarrow{\varphi_i} \mathscr{E} \to \mathscr{E}/\mathscr{E}_j$ vanishes. 
	Then uniqueness follows from that of the slope filtration for $F$-isocrystals. 
	
	We reduce to checking the above claim at the fiber of each closed point $x\in |X|$. 
	We may assume there exists a finite extension $k'/k$ of degree $a$ such that $x\simeq \Spec(k')$ as in \S~\ref{sss:partial Frob slope}. 
	Then the fiber of the Frobenius structure $\varphi_x^a$ is a linear automorphism of $\mathscr{E}_x$, and its generalized eigenspace decomposition is a refinement of the filtration $\mathscr{E}_{0,x}\subset\cdots \subset \mathscr{E}_{\ell,x}$. 
	Since each partial Frobenius structure $\varphi_{i,x}$ commutes with $\varphi_{x}^a$, $\varphi_{i,x}$ preserves each generalized eigenspace of $\varphi_x^a$. 
	Then the assertion follows. 

	(ii) By assertion (i), we may reduce to the case where $\mathscr{E}$ is diagonally unit-root after twisting. 
	In this case, assertion (ii) follows from Corollary \ref{c:constant partial Frob}. 
\end{proof}

\subsection{A variant of Crew's theorem for unit-root convergent $\Phi$-isocrystals}
\begin{secnumber} \label{sss:CXPhi}
	We denote by $\mathcal{C}(X,\Phi)$ the category of objects $(T,\{F_{ \{i\}}\}_{i=1}^n)$ consisting of a finite \'etale morphism $T\to X=X_1\times_k\cdots \times_k X_n$ and isomorphisms $F_{ \{i\}}\colon T\times_{X,F_i}X \xrightarrow{\sim} T$ commuting with each other, i.e. $F_{ \{i\}}\circ F_{ \{j\}}\times_{X,F_j}X= F_{ \{j\}}\circ F_{ \{i\}}\times_{X,F_i}X$, whose composition is the relative Frobenius morphism $F_{T/X}$ of $T$ over $X$.
	A morphism in this category is a morphism above $X$ compatible with the $F_{ \{i\}}$. 
	Note that an object of $\mathcal{C}(X,\Phi)$ is equivalent to the data of a pair $(T, \{\varphi_{T,i}\}_{i=1}^n)$ consisting of a finite \'etale morphism $T\to X$ together with morphisms $\varphi_{T,i}\colon T\to T$ above $F_i$, commuting with each other, whose composition is the Frobenius morphism $F_T$. 

	This category is a Galois category and we denote by $\pi_1^{\et}(X,\Phi,\overline{x})$ the Galois group defined by the fiber functor associated to $\overline{x}$.
	We have the following equivalence:

	\begin{itemize}
		\item the category of continuous actions of $\pi_1(X,\Phi,\overline{x})$ on finite sets;

		\item the category of locally constant constructible sheaves $\mathbb{L}$ of $X_{\et}$, equipped with a partial Frobenius structure $\phi_i\colon F_i^*(\mathbb{L})\xrightarrow{\sim} \mathbb{L}$, commuting with each other, whose composition is the Frobenius correspondence $\phi_{\mathbb{L}}$.  
	\end{itemize}
\end{secnumber}

\begin{prop} \label{P:Crew partial Frob}
	There is a canonical equivalence between the category of continuous representations of $\pi_1^{\et}(X,\Phi,\overline{x})$ on finite-dimensional $K$-vector spaces and the full subcategory $\PhiIsoc^{\ur}(X/K)$ of $\PhiIsoc(X/K)$ consisting of unit-root convergent $\Phi$-isocrystals. 
\end{prop}

\begin{secnumber}
	Let $\pi$ be a uniformizer of $\mathscr{O}_K$. 
	We set $R:=\mathscr{O}_K$ and $R_n:=\mathscr{O}_K/\pi^n \mathscr{O}_K$ for $n\ge 1$. 
	We may reduce to the case where each $X_i$ is affine and admits a smooth formal lifting $\XX_i$ to $R$ and a lifting $F_{\XX_i}\colon \XX_i\to \XX_i$ of the Frobenius $F_{X_i}$. We set $\XX:=\prod_{\rW(k)} \XX_i$ and denote by $F_{i,\XX}\colon \XX\to \XX$ the product of $F_{\XX_i}$ and the identity maps on other components.  
	
	A \textit{unit-root $\Phi$-lattice} on $\XX/R$ is a locally free $\mathscr{O}_{\XX}$-module of finite rank together with isomorphisms $\varphi_i\colon F_{i,\XX}^*(M)\xrightarrow{\sim} M$ commuting with each other. 
	We first establish an equivalence following \cite{crew-f}:
	\begin{equation}
		\iota\colon \Rep_{R}^{\cont}(\pi_1(X,\Phi,\overline{x})) \xrightarrow{\sim} \{\textnormal{Unit-root $\Phi$-lattices on $\XX/R$}\}.
		\label{eq:equiv integral rep}
	\end{equation}

	Let $V$ be a continuous $R$-linear representation of $\pi_1^{\et}(X,\Phi,\overline{x})$. 
	For $n\ge 1$, let $G_n$ be the image of $\pi_1^{\et}(X,\Phi,\overline{x})$ in $\GL(V/\pi^nV)$, which corresponds to a Galois cover $Y^{(n)}\to X$ and $\pi_n\colon \YY_n^{(n)}\to \XX_n$ its lifting to $R_n$. 
	Moreover, each isomorphism $F_{ \{i\}}$ uniquely lifts to an isomorphism $F_{ \{i\}}\colon F_{i,\XX}^*(\YY_n^{(n)})\xrightarrow{\sim} \YY_n^{(n)}$, which commute with each other. As in \cite{crew-f}, we consider a locally free $\mathscr{O}_{\XX_n}$-module (resp. $\mathscr{O}_{\XX}$-module)
	\[
	M_n:= \pi_{n,*}(\mathscr{O}_{\YY_n^{(n)}})\otimes_{R_n[G_n]}V/\pi^n V, \quad M:=\varprojlim M_n. 
	\]
	Then the isomorphism $F_{ \{i\}}$ induces an isomorphism $F_{i,\XX}^*(M_n)\xrightarrow{\sim} M_n$ and gives rise to a unit-root partial Frobenius structure
	\[
\varphi_i\colon F_{i,\XX}^*(M)\xrightarrow{\sim} M.
	\]
	In this way, we obtain the functor $\iota$ as in \eqref{eq:equiv integral rep}.
	
	Conversely, given a unit-root $\Phi$-lattice $(\mathscr{E},\varphi_i)$ on $\XX/R$, the associated unit-root $F$-lattice $(\mathscr{E},\varphi)$ defines a lisse $R$-sheaf $\mathbb{L}$ on $X$ by \cite[Theorem 2.2]{crew-f}.
	The partial Frobenius structures $\varphi_i$ define isomorphisms $\phi_i\colon F_i^*(\mathbb{L})\xrightarrow{\sim} \mathbb{L}$ commuting with each other. Then we obtain a continuous $R$-representation of $\pi_1(X,\Phi,\overline{x})$ on $\mathbb{L}_{\overline{x}}$. 
	The above construction is functorial and defines a quasi-inverse of $\iota$. 
\end{secnumber}
\begin{lemma}
   The isomorphism $\varphi_i$ is horizontal with respect to the canonical convergent connection $\nabla$ on $M^{\rig}$ \cite[Proposition 2.3]{crew-f}.
\end{lemma}
\begin{proof}
	Let $\Delta$ be the formal completion of the diagonal map $\XX\to \XX\times_{R}\XX$ and $p_1,p_2\colon \Delta\to \XX$ the natural projections. 
	For $n\ge 1$, we set $\Delta_n:=\Delta\otimes_{R}R_n$. 
	Let $p_i^*\YY_{n}^{(n)}$ be the fiber product $\YY_{n}^{(n)}\times_{\XX_n,p_i}\Delta_n$. 
	The map $\YY_n^{(n)}\to \XX_n \xrightarrow{diag} \Delta_n$ gives rise to formal thickenings $\YY_n^{(n)}\to p_i^*\YY_{n}^{(n)}$. 
	Recall \textit{loc.cit.} that, since $\YY_n^{(n)}\to \XX_n$ is \'etale, there exists a unique isomorphism $p_1^*\YY_n^{(n)}\xrightarrow{\sim} p_2^*\YY_n^{(n)}$ which fits into the following diagram:
	\[
	\xymatrix{
	p_1^* \YY_n^{(n)} \ar[d] \ar[rd]^{\sim} & \YY_n^{(n)} \ar[l] \ar[d] \\
	\Delta_n & p_2^*\YY_n^{(n)} \ar[l]
	}
	\]
	By uniqueness, this isomorphism is compatible with the $G_n$-action, the partial Frobenius maps on $\YY_n^{(n)}$, and the canonical connection on $M$. Then the lemma follows. 
\end{proof}
\begin{secnumber} \textit{Proof of Proposition \ref{P:Crew partial Frob}.}
	By the above lemma, we have a fully faithful functor 
	\[
	\{\textnormal{Unit-root $\Phi$-lattices on $\XX/R$}\}\otimes_{R}K \to
	\{\textnormal{Unit-root $\Phi$-isocrystals on $X/K$}\}
	\]
	It suffices to show that the functor in Proposition \ref{P:Crew partial Frob} is essentially surjective. 

	Let $(\mathscr{M},\varphi_i)$ be a unit-root $\Phi$-isocrystal on $X/K$ and $(\mathbb{L},\phi_i)$ the associated data in Proposition \ref{P:Crew diagonal unitroot}. 
	There exists a $R$-lattice $\mathbb{L}^{\circ}$ of $\mathbb{L}$ such that $\phi_{\mathbb{L}}$ induces an isomorphism $F_X^*(\mathbb{L}^{\circ})\xrightarrow{\sim} \mathbb{L}^{\circ}$. 
	As $\varphi_i$ is unit-root, we claim there exist integers $r,s$ such that for every integer $m\ge 0$, we have 
		\[
		p^r \mathbb{L}^{\circ} \subset \phi_i^m(F_i^{m,*}\mathbb{L}^{\circ}) \subset p^s \mathbb{L}^{\circ}.
		\]
	Indeed, let $x\in|X|$ be a closed point with degree $[k_x:k]=a$. 
	The linear action of $\phi_{i,x}^a$ on the fiber $\mathbb{L}^{\circ}_{\overline{x}}$ is unit-root and satisfies the above properties \cite[Proposition 1.11]{crew-f}. 
	Then the claim follows.

	Following \cite{crew-f}, we consider $\mathbb{L}':=\sum_{m\ge 0} \IM \phi_i^m\colon \mathbb{L}^{\circ}\to p^s \mathbb{L}$ and $\mathbb{L}_i:= \cap_{m\ge 1} \IM \phi_i^m\colon \mathbb{L}'\to \mathbb{L}'$. 
	Then $\phi_i$ induces an isomorphism $F_i^*\mathbb{L}_i\xrightarrow{\sim} \mathbb{L}_i$. 

	By repeating the above argument to each $\phi_i$, we obtain a $R$-lattice $\widetilde{\mathbb{L}}$ of $\mathbb{L}$ such that $\phi_i$ induces an isomorphism $F_i^*(\widetilde{\mathbb{L}})\xrightarrow{\sim} \widetilde{\mathbb{L}}$ for every $i$. 
	Then the data $(\widetilde{\mathbb{L}},\phi_i)$ descend to a representation of $\pi_1^{\et}(X,\Phi,\overline{x})$.
	In view of Proposition \ref{P:Crew diagonal unitroot}, this defines a quasi-inverse of the previous construction. 
	The proposition follows. \hfill \qed
\end{secnumber}

\subsection{Original Drinfeld's lemma} 
	Following an argument of Drinfeld--Kedlaya \cite[Appendix B]{DK17}, we recover the original Drinfeld's lemma from Theorem \ref{t:main thm} and Proposition \ref{P:Crew partial Frob}.

\begin{secnumber}
	We say a representation of $\Rep^{\cont}_K(\pi_1^{\et}(X,\Phi,\overline{x}))$ is \textit{smooth} if the action of $\pi_1^{\et}(X,\Phi,\overline{x})$ factors through a finite quotient. 
	In view of \'etale descent for overconvergent $F$-isocrystals, the canonical functor 
	\begin{equation}
		\Rep^{\smooth}_{K}(\pi_1^{\et}(X,\Phi,\overline{x})) \to \PhiIsoc(X)
		\label{Crew smooth}
	\end{equation}
	factors through $\PhiIsocd(X)$. 
	It induces canonical homomorphisms: 
	\begin{equation}
		\pi_1^{\PhiIsoc}(X)\to \pi_1^{\PhiIsocd}(X)\to \pi_1^{\et}(X,\Phi,\overline{x}).
		\label{composition pi}
	\end{equation}
	Since the functor \eqref{Crew smooth} is fully faithful, the composition is an epimorphism. 
	
	Writing $\Fet(X_i)$ for the category of finite \'etale schemes over $X_i$, the canonical functor
	\[
	\Fet(X_i)\to \mathcal{C}(X,\Phi), \qquad T_i\mapsto \prod_{j\neq i}X_j \times_k T_i
	\]
	induces a canonical homomorphism $\widetilde{p}_i\colon \pi_1^{\et}(X,\Phi,\overline{x})\to \pi_1^{\et}(X_i,\overline{x}_i)$, where $\overline{x}_i:=p_i(\overline{x})$. 
	In view of the construction, the above morphism is compatible with the projection, that is the following diagram commutes:
	\begin{equation}
		\label{commutative projection}
	\xymatrix{
	\pi_1^{\PhiIsocd}(X) \ar[r] \ar[d]_{p_i^{\circ}} & \pi_1^{\et}(X,\Phi,\overline{x}) \ar[d]_{\widetilde{p}_i} \\
	\pi_1^{\FIsocd}(X_i) \ar[r] & \pi_1^{\et}(X_i,\overline{x}_i)
	}
	\end{equation}
\end{secnumber}

\begin{theorem} \label{t:Drinfeld lemma origin}
	\textnormal{(i)} The homomorphisms \eqref{composition pi} induce isomorphisms
	\[
	\pi_0(\pi_1^{\PhiIsoc}(X))\xrightarrow{\sim} \pi_0(\pi_1^{\PhiIsocd}(X))\xrightarrow{\sim} \pi_1(X,\Phi,\overline{x}).		
	\]

	\textnormal{(ii)}
	If we apply $\pi_0$ to the isomorphism \eqref{Tannakian Drinfeld lemma}, we obtain the original Drinfeld's lemma:
	\[
	\pi_1(X,\Phi,\overline{x})\xrightarrow{\sim} \prod_{i=1}^n \pi_1^{\et}(X_i,\overline{x}_i).
	\]
\end{theorem}

\begin{proof}
	(i) 
	We first show that for any finite group $\Gamma$, any tensor functor
	\begin{equation} \label{eq:Drinfeld lemma pi1}
	\Rep_{\bQp}(\Gamma) \to \PhiIsoc(X)
\end{equation}
	factors through $\Rep_{\bQp}^{\cont}(\pi_1(X,\Phi,\overline{x}))$. 
	By Proposition \ref{P:Crew partial Frob}, it suffices to show that objects in the essential image of the above functor are unit-root at each point of $|X|$. 
	This can be shown by a similar argument to \cite[Proposition~B.4.1]{DK17}. 
	
	Moreover, we deduce that the above functor factors through $\Rep_{\bQp}^{\smooth}(\pi_1(X,\Phi,\overline{x}))$ as in \textit{loc.cit.}

	Hence, the kernel of the canonical epimorphism $\pi_1^{\PhiIsoc}(X)\twoheadrightarrow \pi_1(X,\Phi,\overline{x})$ is the neutral component of $\pi_1^{\PhiIsoc}(X)$. 
	We obtain the isomorphism $\pi_0(\pi_1^{\PhiIsoc}(X))\xrightarrow{\sim} \pi_1(X,\Phi,\overline{x})$.

	By a full faithfulness result à la Kedlaya (Proposition \ref{P:fully faithful}), $\pi_1^{\PhiIsocd}(X)\to \pi_1(X,\Phi,\overline{x})$ is an epimorphism. 
	Then by \cite[Lemma~B.7.4]{DK17}, we deduce the isomorphism $\pi_0(\pi_1^{\PhiIsocd}(X))\xrightarrow{\sim} \pi_1(X,\Phi,\overline{x})$. 

	(ii) 
	Assertion (ii) follows from assertion (i), Theorem \ref{t:main thm}, and the commutative diagram \eqref{commutative projection}. 	
\end{proof}
	
	\begin{prop}\label{P:fully faithful}
		\textnormal{(i)} The canonical functor $\iota_{\Phi}\colon \PhiIsocd(X/K)\to \PhiIsoc(X/K)$ is fully faithful.

		\textnormal{(ii)} 
			For $i=1,\dots,n$, let $U_i\subset X_i$ be an open immersion with dense image and set $U:= \prod_{i=1}^n U_i$. 
	Then the following restriction functors are fully faithful:
	\[
	\PhiIsoc(X/K)\to \PhiIsoc(U/K),\quad \PhiIsocd(X/K)\to \PhiIsocd(U/K).
	\]
	\end{prop}
\begin{proof}
	(i) Consider the following diagram:
	\[
	\xymatrix{
	\PhiIsocd(X)\ar[r]^{\iota_{\Phi}} \ar[d]& \PhiIsoc(X) \ar[d] \\
	\FIsocd(X)\ar[r]^{\iota_F} & \FIsoc(X) }
	\]
	The vertical arrows are faithful and $\iota_F$ is fully faithful \cite{Ked04}. 
	Hence it remains to show the fullness of $\iota_{\Phi}$. 
	
	Given two objects $(\mathscr{E},\varphi_i),(\mathscr{E}',\varphi'_i)$ of $\PhiIsocd(X)$ and a morphism $f\colon \mathscr{E}\to \mathscr{E}'$ of $\PhiIsoc(X)$,
	 $f$ extends to a morphism $f^{\dagger}\colon \mathscr{E}\to \mathscr{E}'$ of overconvergent $F$-isocrystals by the full faithfulness of $\iota_F$. 
	Since the canonical functor $\Isocd(X)\to \Isoc(X)$ is faithful, we deduce that $f^{\dagger}$ is compatible with partial Frobenius structures.  

	(ii) By assertion (i), it suffices to prove the assertion for convergent $\Phi$-isocrystals. 
	By \cite[Theorem 5.3]{kedlaya-isocrystals} and a similar argument as in (i), it suffices to show its fullness. 
	
	Given two objects $(\mathscr{E},\varphi_i),(\mathscr{E}',\varphi'_i)$ of $\PhiIsoc(X/K)$ and a morphism $f\colon \mathscr{E}\to \mathscr{E}'$ of $\PhiIsoc(U/K)$. 
	Then $f$ extends to a morphism $g\colon \mathscr{E}\to \mathscr{E}'$ of convergent $F$-isocrystals on $X$. 
	The compatibility between $f$ and partial Frobenius structures follows from \textit{loc.cit.} 
\end{proof}

\subsection{Partial Frobenius slope filtrations}

\begin{theorem} \label{T:partial slope filtration}
Suppose that an object $\mathscr{E}$ of $\PhiIsoc(X/K)$ has a constant diagonal Newton polygon on $|X|$. Then for $i=1,\dots,n$, $\mathscr{E}$ admits a filtration
\[
0 = \mathscr{E}_0^{(i)} \subset \cdots \subset \mathscr{E}_l^{(i)} = \mathscr{E}
\]
in $\PhiIsoc(X/K)$ and an ascending sequence $\mu_1 < \cdots < \mu_l$ of rational numbers such that for $j=1,\dots,l$, the $i$-th partial Frobenius slope of $\mathscr{E}^{(i)}_j/\mathscr{E}^{(i)}_{j-1}$ equals to $\mu_j$. Moreover, the filtration and sequence are both uniquely determined by this condition. (We call it the \emph{$i$-th partial slope filtration} of $\mathscr{E}$.)
\end{theorem}
\begin{lemma} \label{L:diagonally unit-root decomposition}
	Suppose that $\mathscr{E} \in \PhiIsoc(X/K)$ is diagonally unit-root.
	
	\textnormal{(i)} Then there exists a decomposition in $\PhiIsoc(X/K)$
\[
\mathscr{E}\simeq \bigoplus_{d_1,\dots,d_n} \mathscr{E}_{d_1,\dots,d_n}, 
\]
indexed by tuples $(d_1,\dots,d_n) \in \QQ^n$ with $d_1 + \cdots + d_n =
0$, in which
$\mathscr{E}_{d_1,\dots,d_n}$ has the constant partial Frobenius slopes $(d_1,\ldots,d_n)$ on $|X|$. 

\textnormal{(ii)} After taking extension of scalars, each $\Phi$-isocrystal $\mathscr{E}_{d_1,\ldots,d_n}$ is isomorphic to successive extensions of $\Phi$-isocrystals $\mathscr{E}_1\boxtimes \mathscr{E}_2\boxtimes\cdots\boxtimes \mathscr{E}_n$, where $\mathscr{E}_i\in \FIsoc(X_i/K)$ is isoclinic of slope $d_i$.  
\end{lemma}

\begin{proof} 
	(i) Let $(\mathbb{L},\phi_i)$ be the data associated to $\mathscr{E}$ in Proposition \ref{P:Crew diagonal unitroot}.
	Let $x\in |X|$ be a closed point of degree $a$ and $(V=\mathbb{L}_{\overline{x}},\rho)$ the associated $\pi_1(X,\overline{x})$-representation. 
	Since $\phi_{i,\overline{x}}$ commute with each other, $V$ admits a decomposition 
	\[
	V\simeq \bigoplus_{(d_1,\ldots,d_n)} V_{d_1,\ldots,d_n},
	\]
	according to partial Frobenius slopes of $(\phi_{1,\overline{x}},\ldots,\phi_{n,\overline{x}})$. 
	It remains to show that each component is invariant under the $\pi_1(X,\overline{x})$-action. 
	Note that each $\phi_{i,\overline{x}}$ preserves $V_{d_1,\ldots,d_n}$, as then does their composition $\phi_{\mathbb{L},\overline{x}}$. 

	Since $F_X^a$ induces an identity on \'etale fundamental groups $F_X^a\colon\pi_1(X,\overline{x})=\pi_1(X,\overline{x})$, $\phi_{\mathbb{L},\overline{x}}^a$ is an automorphism of the representation $(V,\rho)$. 
	Since $\phi_{\mathbb{L},\overline{x}}^a$ commutes with the action of $\pi_1(X,\overline{x})$ on $V$, we deduce that the $\pi_1(X,\overline{x})$-action also preserves each component $V_{d_1,\ldots,d_n}$. Then assertion (i) follows. 

	(ii) After taking extension of scalars and twisting, we may assume  $\mathscr{F}=\mathscr{E}_{d_1,\ldots,d_n}$ is unit-root. 
	We prove the assertion by induction on $n$. 
	Let $\overline{x}$ be a geometric point of $X$. 
	By Proposition \ref{P:Crew partial Frob} and Theorem \ref{t:Drinfeld lemma origin}, $\mathscr{E}$ corresponds to a continuous $K$-linear representation $V$ of $\pi_1(X_1,\overline{x})\times\cdots \times \pi_1(X_n,\overline{x})$. 
	Let $W$ be an irreducible sub-$\pi_1(X_n,\overline{x})$-representation of $V$. 
	Then $W$ corresponds to a unit-root $F$-isocrystal $\mathscr{F}_n$ on $X_n$, and $\Hom_{\pi_1(X_n,\overline{x})}(W,V)$, viewed as a $\pi_1(X_1,\overline{x})\times\cdots\times \pi_1(X_{n-1},\overline{x})$-representation, corresponds to a unit-root $\Phi$-isocrystal $\mathscr{E}'$ on $X_1\times_k\cdots\times_k X_{n-1}$. 
	We obtain a monomorphism $\mathscr{E}'\boxtimes \mathscr{F}_n\to \mathscr{F}$. 
	Applying the induction hypothesis to $\mathscr{E}'$, then assertion (ii) follows.  
\end{proof}

\begin{lemma}
	\label{L:split by slope}
	Let $\mathscr{E},\mathscr{E}'$ be two objects of $\PhiIsoc(X/K)$ with constant partial Frobenius slopes $(\mu_i)_{i=1}^n,(\mu_i')_{i=1}^n$ respectively. 

	\textnormal{(i)} If $\Hom_{\PhiIsoc(X/K)}(\mathscr{E}',\mathscr{E})\neq 0$, then $\mu_i=\mu_i'$ for every $i=1,\dots,n$. 

	\textnormal{(ii)} If $\Ext_{\PhiIsoc(X/K)}(\mathscr{E}',\mathscr{E})\neq 0$, then $\mu_i\ge \mu_i'$ for every $i=1,\dots,n$. 
\end{lemma}
\begin{proof}
	(i) Let $i$ be an integer, $X':=\prod_{j\neq i} X_j$ and $k'$ a closed point of $X'$.
	We apply the exact and faithful functor \eqref{eq:partial Frob to FIsocd} to the pullback along $\Spec(k')\to X'$ to obtain convergent $F$-isocrystals on $X_{i,k'}/K'$. 
	Then the assertion follows from the corresponding assertion for convergent $F$-isocrystals. 

	(ii) We will deduce this vanishing result of $\Ext$ from the case of $F$-isocrystals, which is known (it follows from the existence of the usual slope filtration). 
	We set $\mathscr{F}:=\FHom(\mathscr{E}',\mathscr{E})$. 
	By Lemma \ref{L:diagonally unit-root decomposition}(ii), we may reduce to the case where $\mathscr{F}\simeq \mathscr{F}_1\boxtimes\cdots\boxtimes \mathscr{F}_n$, with $\mathscr{F}_i\in \FIsoc(X_i/K)$ after taking extension of scalars. 
	Suppose $\mu_i'>\mu_i$ for some integer $i\in [1,n]$. 
	We consider the extension group $\Ext^1_{F_i\textnormal{-}\Isoc(X/K)}(\mathscr{E}',\mathscr{E})$, which fits into the following diagram:
	\begin{equation} \label{eq:ext}
	0\to \rH^0(X,\mathscr{F})/(\varphi_i-\id)\to \Ext^1_{F_i\textnormal{-}\Isoc(X/K)}(\mathscr{E}',\mathscr{E}) \to \rH^1(X,\mathscr{F})^{\varphi_i=\id}\to 0.
\end{equation}

	Since the slope of $\mathscr{F}_i$ is $>0$, we obtain the vanishing of $\rH^0(X_i,\mathscr{F}_i)^{\varphi_i=\id},\rH^0(X_i,\mathscr{F}_i)/(\varphi_i-\id)$, and $\rH^1(X_i,\mathscr{F}_i)^{\varphi_i=\id}$; note that the last case follows from the assertion (ii) for convergent $F$-isocrystals. 
	We conclude that the middle term of \eqref{eq:ext} vanishes by the Künneth formula \cite[Lemma 4.5]{Abe14}. Then the assertion follows.  
\end{proof}
\begin{secnumber} \textit{Proof of Theorem \ref{T:partial slope filtration}.} 
	By Theorem \ref{t:slope filtration}, each partial Newton polygon of $\mathscr{E}$ is also constant. 
	Then we apply Lemma \ref{L:diagonally unit-root decomposition} to the successive quotients of the diagonal slope filtration; this yields a filtration in which each successive quotient has the property that every $i$-th partial Frobenius slope has a fixed value, but these values may not occur in ascending order. However, using Lemma \ref{L:split by slope} we can reorder the successive quotients to enforce this condition. \hfill \qed
\end{secnumber}

\subsection{Drinfeld's lemma for convergent $\Phi$-isocrystals}

Recall that the category $\FIsoc(X/K)\otimes_K \bQp$ (resp. $\PhiIsoc(X/K)\otimes_K\bQp$) is denoted by $\FIsoc(X)$ (resp. $\PhiIsoc(X)$) and is neutral Tannakian over $\bQp$. 
With notations of \S~\ref{sss:Tannakian gps}, we show a version of Drinfeld's lemma for convergent $F$-isocrystals as in Theorem \ref{t:main thm}. 

\begin{theorem}
	\label{t:Drinfeld lemma conv}
	Assume $X_i$ is a smooth geometrically connected $k$-variety. 
	The following canonical homomorphism is an isomorphism:
	\begin{equation} \label{eq:Drinfeld lemma conv}
		\prod_{i=1}^n p_i^{\circ}\colon
		\pi_1^{\PhiIsoc}(X) \xrightarrow{\sim} \prod_{i=1}^n \pi_1^{\FIsoc}(X_i) . 
	\end{equation}	
\end{theorem}

\begin{prop}
	\label{p:pullback subquotients}
	The pullback functor:
	\[
		p_i^*\colon\FIsoc(X_i)\to \PhiIsoc(X) 
	\]
	is fully faithful and its essential image is closed under subquotients.
\end{prop}

\begin{proof}
	We may assume $i=n$ and set $X':=\prod_{i=1}^{n-1}X_i$.

	(i) We first prove the full faithfulness. 
	By Proposition \ref{P:fully faithful}, we may assume each $X_i$ is affine and admits a smooth formal lifting $\XX_i$ over $\mathscr{O}_K$. 
	Let $(\mathscr{E},\varphi)$ be an object of $\FIsoc(X_n)$. Since each $X_i$ is smooth and geometrically connected, the canonical morphism
	\[
	\rH^0(X_n,\mathscr{E})\xrightarrow{\sim} \rH^0(X,p_n^*(\mathscr{E}))
	\]
is an isomorphism by the Künneth formula for the sheaf of differential operators \cite[Lemma 4.5]{Abe14} and for quasicoherent modules. 
	Then we deduce a canonical isomorphism: 
	\[
	\rH^0(X_n,\mathscr{E})^{\varphi=\id}\to \rH^0(X,p_n^*(\mathscr{E}))^{\Phi=\id}.
	\]
	The full faithfulness follows from the above isomorphism applied to internal homomorphisms.	

	(ii) Next we treat the second assertion. 
	Let $\mathscr{E}$ be an object of $\FIsoc(X_n)$ and $p_n^*(\mathscr{E})\to \mathscr{F}$ a surjection in $\PhiIsoc(X)$. 
	By enlarging $k$ \eqref{eq:ext sca}, we may assume there exists a $k$-point $x$ of $X'$. 
	We will show that the surjection $\mathscr{E}\to \mathscr{G}:=(x\times_k \id_{X_n})^*(\mathscr{F})$ over $X_n$ is isomorphic to $p_n^*(\mathscr{E})\to \mathscr{F}$ after pullback via $p_n^*$. 

	Let $U$ be a dense open subset of $X$ such that the diagonal Newton polygon of $\mathscr{F}$ is constant, which is preserved by partial Frobenius morphisms. 
	By Drinfeld's lemma for open immersions \cite[Lemma~9.2.1]{lau}, \cite[Theorem~4.3.6]{kedlaya-aws} and Proposition \ref{P:fully faithful}, we may assume that the diagonal Newton polygons of $p_n^*(\mathscr{E})$ and $\mathscr{F}$ are constant on $|X|$ after shrinking $X_i$.
	Then the set of each partial Newton polygon of $\mathscr{F}$ is also constant on $|X|$. 
	To show the claim, we may reduce to the case where the Frobenius slopes of $\mathscr{E}$ are constant on $X_n$ by applying Theorem \ref{t:slope filtration} to $p_n^*(\mathscr{E})\to \mathscr{F}$. 
	Moreover, we may assume $\mathscr{E}$ is unit-root by twisting. 
	Then $\mathscr{F}$ is also a unit-root $\Phi$-isocrystal. 
	By Proposition \ref{P:Crew partial Frob} and Theorem \ref{t:Drinfeld lemma origin}, for a geometric point $\overline{y}$ of $X$, the action of $\prod_{i=1}^{n-1}\pi_1(X_i,\overline{y})$ on $p_n^*(\mathscr{E})_{\overline{y}}$ is trivial. 
	Then the same holds for $\mathscr{F}$ and the assertion follows. 
\end{proof}

\begin{prop} \label{p:pushforward conv}
	\textnormal{(i)} The functor $p_i^*:\FIsoc(X_i)\to \PhiIsoc(X)$ admits a right adjoint $p_{i,*}$. 
	Moreover, the canonical morphism $\id\to p_{i,*}p_i^*$ is an isomorphism and $p_i^*p_{i,*}\to \id$ is injective. 

	\textnormal{(ii)} The functor $p_{i,*}$ commutes with the base change of $g\colon Y\to X_i$. 
\end{prop}

\begin{proof}
	(i) Let $\mathscr{E}$ be an object of $\PhiIsoc(X)$. By Proposition \ref{p:pullback subquotients}, the collection of subobjects of $\mathscr{E}$ which belong to the essential image of $p^*_i$ has a maximal element $\mathscr{F}$. 
	Then we define $p_{i,*}(\mathscr{E})$ to be the object of $\FIsoc(X_i)$ of which $\mathscr{F}$ is a pullback. 

	Let $f:\mathscr{E}_1\to \mathscr{E}_2$ be a morphism of $\PhiIsoc(X)$. 
	Let $\mathscr{G}$ be the image of $p_i^*(p_{i,*}(\mathscr{E}_1))$ in $\mathscr{E}_2$. 
	By Proposition \ref{p:pullback subquotients}, $\mathscr{G}$ belongs to the essential image of $p_i^*$ and the functoriality of $p_{i,*}$ follows. 
	The second assertion is clear. 

	(ii) By Lemma \ref{P:fully faithful}, we may assume that the diagonal Newton polygon of $\mathscr{E}$ is constant on $|X|$. 
	By Theorem \ref{t:slope filtration}, we may moreover assume that $\mathscr{E}$ is diagonally unit-root after twisting the $i$-th partial Frobenius structure. 
	Then $p_{i,*}(\mathscr{E})$ is unit-root and $p_i^*(p_{i,*}(\mathscr{E}))$ lies in the unit-root component $\mathscr{E}_{(0,\ldots,0)}$ of $\mathscr{E}$ (cf. Lemma \ref{L:diagonally unit-root decomposition}). 
	Moreover, the unit-root $F$-isocrystal $p_{i,*}(\mathscr{E})$ over $X_i$ can be identified with the $\prod_{j\neq i}\pi_1(X_j,\overline{x})$-invariant part of the $\pi_1(X,\Phi,\overline{x})$-representation associated to $\mathscr{E}_{(0,\ldots,0)}$ via Proposition \ref{P:Crew partial Frob}. 
	The base change property follows from this identification in view of Theorem \ref{t:Drinfeld lemma origin}. 
\end{proof}

\begin{secnumber}
	\textit{Proof of Theorem \ref{t:Drinfeld lemma conv}.}
	The proof is similar to that of \S~\ref{pf main thm}(a). 
	We keep the notation of \textit{loc.cit.} 
	It suffices to show the exactness of the following sequence:
	\[
	\pi_1^{\PhiIsoc}(X',\Phi') \xrightarrow{u^{\circ}}
	\pi_1^{\PhiIsoc}(X,\Phi) \xrightarrow{p_1^{\circ}}
	\pi_1^{\FIsoc}(X_1).
	\]
	We verify the exactness using Theorem \ref{exact criterion}. 

	Condition (i) follows from Proposition \ref{p:pullback subquotients}. 
	Conditions (ii) and (iii.c) can be verified in the same way as in \S~\ref{pf main thm}(a).
	Finally, as in \textit{loc.cit.}, we verify conditions (iii.a) and (iii.b) using functor $p_{i,*}$ and its base change property Proposition \ref{p:pushforward conv}. This finishes the proof. \hfill \qed 
\end{secnumber}

\begin{remark}
	The functor $\PhiIsocd(X)\to \PhiIsoc(X)$ induces a canonical morphism $\pi_1^{\PhiIsoc}(X)\to \pi_1^{\PhiIsocd}(X)$, which fits into the following diagram:
	\[
	\xymatrix{
	\pi_1^{\PhiIsoc}(X) \ar[r]^-{\prod_{i=1}^n p_i^{\circ}}_-{\sim}\ar[d] & \prod_{i=1}^n \pi_1^{\FIsoc}(X_i) \ar[d] \\
	\pi_1^{\PhiIsocd}(X) \ar[r]^-{\prod_{i=1}^n p_i^{\circ}}_-{\sim}& \prod_{i=1}^n \pi_1^{\FIsocd}(X_i)
	}
	\]
	In view of the above diagram, the pushforward functor for convergent $\Phi$-isocrystals is compatible with the one for overconvergent $\Phi$-isocrystals.
\end{remark}

\begin{cor}
	Any object $\mathscr{E}$ of $\Isod(X,\Phi)$ (resp. $\PhiIsocd(X)$, resp. $\PhiIsoc(X)$) is a subobject (or quotient) of an object of the form $\boxtimes_{i=1}^n \mathscr{E}_i$, where each $\mathscr{E}_i$ is an object of $\Isod(X_i)$ (resp. $\FIsocd(X_i)$, resp. $\FIsoc(X)$).
\end{cor}
\begin{proof}
	It suffices to show the assertion about the quotient. By induction, we can reduce to the case where $n=2$. 
	We set $G:=\pi_1^{\Isod}(X,\Phi), G_i:=\pi_1^{\Isocd}(X_i)$.
	Let $V$ be a representation of $G\simeq G_1\times G_2$ corresponding to $\mathscr{E}$. We have morphisms of $G_1\times G_2$-representations:
	\[
	(V\otimes V^{\vee})^{G_2}\otimes V \hookrightarrow V\otimes V^{\vee}\otimes V \xrightarrow{\id_V\otimes \ev} V,
	\]
	where $G_1$ only acts at the first component of $V\otimes V^{\vee}\otimes V$ and $G_2$ acts diagonally. 
	By composing with $V\xrightarrow{\coev\otimes \id} V\otimes V^{\vee}\otimes V$, we see that the above composition is surjective. This finishes the proof. 	
\end{proof}

\section{A variant for $\widehat{\Pi}^{\mot}$} \label{s:motivic}

In \cite{Dri18}, Drinfeld considered the $\ell$-adic pro-semisimple completion of the \'etale fundamental group and showed this object is independent of $\ell$ up to conjugation by elements of the neutral component. 
In this section, we incorporate Drinfeld's lemma in that setting.

	Throughout \S\ref{s:motivic}, we assume $(k,\sigma)=(\mathbb{F}_q,\id_{\mathscr{O}_K})$, let $X_i$ be smooth geometrically connected $k$-varieties for $i=1,\ldots,n$ and set $X:=\prod_{i=1}^n X_i$. 

	\subsection{On semisimplicity of the monodromy group}
In this subsection, we discuss the semisimplicity of the monodromy group for an overconvergent $\Phi$-isocrystal following \cite{WeilII,Cr92}. 

An \textit{$n$-twist} is a $\Phi$-isocrystal on $\Spec(k)/K$ of rank 1.
Let $\chi:=(\chi_i)_{i=1}^n$ be an $n$-twist and $\mathscr{E}$ an overconvergent $\Phi$-isocrystal. 
We denote by $\mathscr{E}(\chi)$ the tensor product $\mathscr{E}\otimes f^*(\chi)$, where $f:X\to \Spec(k)$ is the structure morphism. 
We first prove a result in the rank one case, generalizing a result of Abe \cite[Lemma 6.1]{Abe15}.

\begin{lemma} \label{p:rankone}
	Let $\mathscr{E}$ be an overconvergent $\Phi$-isocrystal of rank one over $X$. 
	Then there exists an $n$-twist $\chi$ such that $\mathscr{E}(\chi)$ has finite order. 
\end{lemma}

\begin{proof}
	After twisting, we may assume $\mathscr{E}$ is unit-root.
	Let $\rho$ be the character of $\pi_1(X,\Phi,\overline{x})$ associated to $\mathscr{E}$ (Proposition \ref{P:Crew partial Frob}) and $\rho_i$ the restriction of $\rho$ to $\pi_1(X_i,\overline{x})$ (Theorem \ref{t:Drinfeld lemma origin}).  
	We may shrink each $X_i$ and in particular, we may assume that $X_i$ admits a smooth compactification $X_i\to \overline{X}_i$ such that the complement is a simple normal crossing divisor. 
	Using \cite[Theorem 4.3]{Shi11}, \cite[Theorem 2]{KL} and the same argument of \cite[Corollary 4.13]{crew-f}, one can show that some power $\mathscr{E}^{\otimes N}$ is geometrically trivial, i.e. its underlying overconvergent isocrystal is trivial. 
	Then $\rho^{\otimes N}$ factors through the quotient $\pi_1(X,\Phi,\overline{x})\to \widehat{\mathbb{Z}}^n$ and each $\rho_i^{\otimes N}$ is also geometrically trivial.
	Hence the unit-root convergent $F$-isocrystal $\mathscr{E}_i$ over $X_i$ associated to $\rho_i$ is overconvergent. 
	We may take suitable twists $\chi_i$ such that each $\mathscr{E}_i(\chi_i)$ has finite order \cite[Lemma 6.1]{Abe15}. Then the assertion follows.  
\end{proof}

This allows us to conclude the following corollary.
\begin{cor}
	Let $\mathscr{E}$ be an overconvergent $\Phi$-isocrystal over $X$. 
	After taking a finite extension of $K$, there exist an integer $m\ge 1$, $n$-twists $\chi_i$ for $1\le i\le m$, and a decomposition of the semisimplification of $\mathscr{E}$:
	\[
	\mathscr{E}^{\semis}\simeq \bigoplus_{i=1}^m \mathscr{F}_i(\chi_i),
	\]
	where $\mathscr{F}_i$ is an irreducible overconvergent $\Phi$-isocrystal of finite order determinant for each $i$. 
\end{cor}

\begin{defn}
We say $\mathscr{E}$ is \textit{untwisted} if we can choose each $\chi_i$ to be the trivial $\Phi$-isocrystal on $\Spec(k)$ in the above corollary. 
\end{defn}
\begin{secnumber}
	Let $\mathscr{E}$ be an object of $\PhiIsocd(X)$. 
	We denote by $G_{\geo}(\mathscr{E})$ (resp. $G_{\arith}(\mathscr{E})$) its geometric (resp. arithmetic) monodromy group over $\bQp$, which is defined by the Tannakian full subcategory of $\PhiIsocd(X)$ (resp. $\Isocd(X)$) whose objects are subquotients of $\mathscr{E}^{\otimes m}\otimes \mathscr{E}^{\vee, \otimes n}$ for some $m,n$.  
	Recall that the radical of $G_{\geo}(\mathscr{E})$ is unipotent  \cite[Theorem 4.9]{Cr92}. 

	Given a field $E$, by a \textit{semisimple} group (resp. \textit{reductive} group) over $E$, we mean an algebraic group of finite type over $E$ whose neutral component is semisimple (resp. reductive). 
	
	By a similar argument to \cite[Theorem 3.4.7]{DZ20}, we conclude the following result for $\Phi$-isocrystals. 
\end{secnumber}

\begin{prop} \label{p:radical unipotent}
	The following properties are equivalent:

	\begin{itemize}
		\item[(i)] The neutral component of $G_{\arith}(\mathscr{E})/G_{\geo}(\mathscr{E})$ is unipotent. 

		\item[(ii)] The radical of $G_{\arith}(\mathscr{E})$ is unipotent. 

		\item[(iii)] The object $\mathscr{E}$ is untwisted. 
	\end{itemize}
\end{prop}

\begin{cor} \label{c:semisimple monodromy}
	Let $\mathscr{E}$ be a semisimple object of $\PhiIsocd(X)$.

	\textnormal{(i)} The geometric monodromy group $G_{\geo}(\mathscr{E})$ is semisimple. 

	\textnormal{(ii)} If the determinant of each irreducible component of $\mathscr{E}$ has finite order, then $G_{\arith}(\mathscr{E})$ is semisimple. 	
\end{cor}

\begin{proof}
	Assertion (i) follows from the fact that the radical of $G_{\geo}(\mathscr{E})$ is unipotent and the argument of \cite[Corollaire 1.3.9]{WeilII}.
	Assertion (ii) can be shown in a similar way using Proposition \ref{p:radical unipotent}. 
\end{proof}

\subsection{Pro-semisimple completion of the fundamental group of a smooth variety with partial Frobenius}

\begin{secnumber}
	Let $E$ be an algebraically closed field. 
	We refer to \cite[\S 2.1.1]{Dri18} for the notion of pro-semisimple groups and pro-reductive groups. 
	Following \cite[\S 1.2.3]{Dri18}, $\Pred(E)$ denotes the groupoid whose objects are pro-reductive groups over $E$ and whose morphisms are as follows: a morphism $G_1\to G_2$ is an isomorphism of group schemes $G_1\xrightarrow{\sim} G_2$ defined up to composing with automorphisms of $G_2$ of the form $x\mapsto gxg^{-1}$, $g\in G_2^{\circ}$. 
	Let $\Pss(E)\subset \Pred(E)$ be the full subcategory formed by pro-semisimple groups. 

For any pro-algebraic group $G$ over $E$, we denote by $G^{\red}$ (resp. $G^{\semis}$) its pro-reductive (resp. pro-semi-simple) quotient. 
\end{secnumber}

\begin{secnumber}
	Fix an algebraic closure $\bQQ$ of $\QQ$. Let $\lambda$ be a non-Archimedean place of $\bQQ$, $\ell$ the prime that $\lambda$ divides $\ell$, $\bQla$ the direct limit of $E_{\lambda}$ for subfields $E\subset \bQQ$ finite over $\QQ$. 
	Let $(\widetilde{X},\Phi)$ (resp. $\widetilde{X}_i$) be the universal cover of the category $\mathcal{C}(X,\Phi)$ \eqref{sss:CXPhi} (resp. the universal \'etale cover of $X_i$) and $\Pi(X,\Phi)$ the automorphism group of $(\widetilde{X},\Phi)/X$, which is isomorphic to $\pi_1(X,\Phi,\overline{x})$ after choosing a base point.  
	When $n=1$, we write $\Pi(X)$ for $\Pi(X,\Phi)$.
	
	Following Drinfeld, we define a pro-semisimple group $\HPi_{\lambda}(X,\Phi)$ over $\bQla$ as follows:
	\begin{itemize}
		\item[(i)] When $\ell=p$, we denote by $\ccT_{\lambda}(X,\Phi)$ the full subcategory of $\PhiIsocd(X)$ of semisimple objects $M$ such that the determinant of each irreducible component of $M$ has finite order. 
	This category is a Tannakian full subcategory of $\PhiIsocd(X)$ and we denote by $\HPi_{\lambda}(X,\Phi)$ its Tannakian group. 
	By Corollary \ref{c:semisimple monodromy}, the quotient $\pi_1^{\PhiIsocd}(X)\to \HPi_{\lambda}(X,\Phi)$, induced by the inclusion functor, identifies with the pro-semisimple quotient of $\pi_1^{\PhiIsocd}(X)$. 

\item[(ii)] When $\ell\neq p$, we define $\HPi_{\ell}(X,\Phi)$ to be the $\ell$-adic pro-semisimple completion of $\Pi(X,\Phi)$ and set $\HPi_{\lambda}(X,\Phi)=\HPi_{\ell}(X,\Phi)\otimes_{\Ql}\bQla$ \cite[\S 1.2.1]{Dri18}. 
	
	There is an equivalent definition: we define a full subcategory $\ccT_{\lambda}(X,\Phi)$ of the category of lisse $\bQla$-sheaves on $X$ with a partial Frobenius structure, formed by semisimple objects $M$ such that the determinant of each irreducible component of $M$ has finite order as in (i). 
	It is a Tannakian category over $\bQla$ and the associated Tannakian group is isomorphic to $\HPi_{\lambda}(X,\Phi)$ (cf. \cite[\S 3.6]{Dri18} in the case $n=1$). 
\end{itemize}

The embedding $\bQQ\to \bQla$ induces an equivalence \cite[Proposition 2.2.5]{Dri18}:
	\begin{equation}
		\Pss(\bQQ)\xrightarrow{\sim} \Pss(\bQla).
		\label{eq:equi alg closed}
	\end{equation}
	We denote by $\HPi_{(\lambda)}(X,\Phi)$ the object of $\Pss(\bQQ)$ associated to $\HPi_{\lambda}(X,\Phi)$ by above equivalence.
	When $n=1$, we omit $\Phi$ from the above notations. 

	By Drinfeld's lemma, there exists a canonical isomorphism over $\bQla$:
	\[
\HPi_{\lambda}(X,\Phi)\xrightarrow{\sim} \prod_{i=1}^n \HPi_{\lambda}(X_i).
	\]
	When $\lambda$ divides $p$, it follows from taking the pro-semisimple quotient of \eqref{Tannakian Drinfeld lemma}; when $\lambda$ does not divide $p$, it follows from \eqref{eq:Drinfeld lemma pi1} and taking the $\ell$-adic pro-semisimple completion. 
	Via the equivalence \eqref{eq:equi alg closed}, we obtain an isomorphism over $\bQQ$: 
	\begin{equation}
		\HPi_{(\lambda)}(X,\Phi)\xrightarrow{\sim} \prod_{i=1}^n \HPi_{(\lambda)}(X_i). 
		\label{eq:HPi prod}
	\end{equation}
\end{secnumber}

\begin{secnumber} \label{sss:tildePi}
	We review some constructions in \cite{Dri18}. 
	For each $i$, let $|\widetilde{X}_i|^\circ$ be the set of closed points of $\widetilde{X}_i$. We have a canonical $\Pi(X_i)$-equivariant map \cite[(1.1)]{Dri18}:
	\[
	|\widetilde{X}_i|^\circ \to \Pi(X_i),\quad \widetilde{x}\mapsto F_{\tilde{x}},
	\]
	where $F_{\tilde{x}}$ denotes the geometric Frobenius automorphism at $\widetilde{x}$.

	The set $\Pi(X_i)$ contains a dense subset $\Pi_{\Fr}(X_i)$ formed by the elements $F_{\tilde{x}}^n$ for $\tilde{x}\in |\widetilde{X}_i|^\circ$ and $n\in \mathbb{Z}_{\ge 0}$. 
	The group $\Pi(X_i)$ acts on $\Pi_{\Fr}(X_i)$ by conjugation. 
	We denote by $\widetilde{\Pi}_{\Fr}(X_i):=\mathbb{Z}_{\ge 0}\times |\widetilde{X}|^\circ$. One has the canonical $\Pi(X_i)$-equivariant surjection:
	\begin{equation} \label{eq:Pi tilde}
	\widetilde{\Pi}_{\Fr}(X_i)\twoheadrightarrow \Pi_{\Fr}(X_i),\quad (n,\tilde{x}) \mapsto F_{\tilde{x}}^n.
\end{equation}
	
	For any pro-reductive group $G$, let $[G]$ denote the GIT quotient of $G$  by the conjugation action of the neutral component $G^{\circ}$. 
	By Theorem \ref{t:Drinfeld lemma origin}, we have an adjoint action of $\Pi(X,\Phi)$ on $[\HPi_{(\lambda)}(X,\Phi)]$ and a canonical $\Pi(X,\Phi)$-equivariant map $[\HPi_{(\lambda)}(X,\Phi)]\twoheadrightarrow \Pi(X,\Phi)$. 

	The map \eqref{eq:Pi tilde} has a canonical lift to a $\Pi(X_i)$-equivariant map (c.f. \cite[\S 7.2.4, 7.3.5]{Dri18})
	\begin{equation}
		\widetilde{\Pi}_{\Fr}(X_i)\to [\HPi_{(\lambda)}(X_i)](\bQla),\quad (n,\tilde{x})\mapsto F_{\tilde{x}}^n. 
		\label{eq:Pitilde Hat}
	\end{equation}
	By the companion theorem \cite{abe-companion, AE16, kedlaya-companionI, kedlaya-companionII}, the above map factors as \cite[Corollary 7.4.2]{Dri18}
	\[
	\widetilde{\Pi}_{\Fr}(X_i)\twoheadrightarrow \Pi_{\Fr}(X_i) \to [\HPi_{(\lambda)}(X_i)](\bQQ) \hookrightarrow [\HPi_{(\lambda)}(X_i)](\bQla).
	\]
	For any non-Archimedean place $\lambda$ of $\bQQ$, we have a diagram of sets:
	\[
	\Pi_{\Fr}(X_i)\to [\HPi_{(\lambda)}(X_i)](\bQQ) \twoheadrightarrow \Pi(X_i).
	\]
	
	Via the isomorphism \eqref{eq:HPi prod}, we obtain a diagram of sets:
	\begin{equation}
		\prod_{i=1}^n \Pi_{\Fr}(X_i) \to [\HPi_{(\lambda)}(X,\Phi)](\bQQ)\twoheadrightarrow \Pi(X,\Phi). 
		\label{eq:Frob seq}
	\end{equation}
	
The result of \cite[Theorem 1.4.1]{Dri18} can be generalized as follow. 
\end{secnumber}

\begin{prop} \label{P:l-independence}
	Let $\lambda, \lambda'$ be non-Archimedean places of $\bQQ$. 
	There exists a unique isomorphism 
	\[
	\HPi_{(\lambda)}(X,\Phi)\xrightarrow{\sim} \HPi_{(\lambda')}(X,\Phi)
	\]
	in the category $\Pss(\bQQ)$ which sends the diagram \eqref{eq:Frob seq} to the corresponding diagram for $\HPi_{(\lambda')}$. 
	\label{t:HPi unique}
\end{prop}

\begin{proof}
	When $n=1$ (i.e. we don't consider partial Frobenius structures), Drinfeld proved the above result for higher-dimensional smooth geometrically connected $k$-varieties if $\lambda,\lambda'$ do not divide $p$ and the result in the curve case when $\lambda$ or $\lambda'$ divides $p$ \cite[Theorem 1.4.1, 5.2.1]{Dri18}. 
	Now we can obtain the full generality of the theorem in the case $n=1$ using the recent breakthrough in the companion theorem for $p$-adic coefficients \cite{AE16, kedlaya-companionI, kedlaya-companionII} over smooth $k$-varieties and the same argument of \cite{Dri18}. 

	In general, the isomorphism $\prod_{i=1}^n \HPi_{(\lambda)}(X_i)\xrightarrow{\sim} \prod_{i=1}^n \HPi_{(\lambda')}(X_i)$, obtained in the case $n=1$, induces an isomorphism $\HPi_{(\lambda)}(X,\Phi)\xrightarrow{\sim} \HPi_{(\lambda')}(X,\Phi)$, which fits into the following diagram:
	\[
	\xymatrix{
	\HPi_{(\lambda)}(X,\Phi)\ar[r]^-{\sim} \ar[d]_{\sim} & \prod_{i=1}^n \HPi_{(\lambda)}(X_i) \ar[d]^{\sim}\\
		\HPi_{(\lambda')}(X,\Phi) \ar[r]^-{\sim}& \prod_{i=1}^n \HPi_{(\lambda')}(X_i)
	}
	\]
	In view of the definition of \eqref{eq:Frob seq}, the required properties and the uniqueness follows from the case $n=1$.  
\end{proof}

\begin{cor}
	The neutral component $\widehat{\Pi}^{\circ}_{(\lambda)}(X,\Phi)$ is simply connected. 
\end{cor}

\begin{proof}
	We may assume that $\lambda$ is coprime to $p$. Then the statement follows from isomorphism \eqref{eq:HPi prod} and \cite[Proposition 3.3.4]{Dri18}. 
\end{proof}

\begin{secnumber}
	Let $\lambda$ be a non-Archimedean place of $\bQQ$. 
	In \cite[\S~6]{Dri18}, Drinfeld defined a pro-reductive group $\widehat{\Pi}^{\mot}_{(\lambda)}$ over $\overline{\mathbb{Q}}$ (independent of $\lambda$ up to unique isomorphisms) and introduced an unconditional definition of motivic Langlands parameters proposed by V. Lafforgue \cite{lafforgue-v}. 
	
	Now we introduce its variant $\HPi^{\mot}_{(\lambda)}(X,\Phi)$ with partial Frobenius. 
	Let $\mathcal{W}_p\subset \overline{\mathbb{Q}}^{\times}$ denote the subgroup of $p$-Weil numbers. 
	For every $n\ge 1$, we set $D_n:=\Hom(\mathcal{W}_p,\bGm^n)$ and $D_{n,\lambda}:=D_n\otimes_{\overline{\mathbb{Q}}}\bQla$ \cite[\S~6.1]{Dri18}. 
	Note that the group $\pi_0(D_n)$ identifies with the group $\Hom(\mu_{\infty}(\bQQ),\bGm^n)\simeq \widehat{\mathbb{Z}}^n$. 

	On the other hand, we have a canonical surjection:
	\[
	\Pi(X,\Phi)\twoheadrightarrow \Gal(\overline{k}/k)^n \simeq \widehat{\mathbb{Z}}^n.
	\]

	Following \cite[\S 6.1]{Dri18}, we define $\widehat{\Pi}_{(\lambda)}^{\mot}(X,\Phi)$ as the fiber product of $\HPi_{(\lambda)}(X,\Phi)$ and $D_{n}$ over $\widehat{\mathbb{Z}}^{n}$. 
	Then we can upgrade the isomorphism \eqref{eq:HPi prod} to be an isomorphism in $\Pred(\overline{\mathbb{Q}})$:
	\begin{equation}
		\widehat{\Pi}_{(\lambda)}^{\mot}(X,\Phi)\xrightarrow{\sim} \prod_{i=1}^n \widehat{\Pi}_{(\lambda)}^{\mot}(X_i).
		\label{eq:Drinfeld lemma mot}
	\end{equation}
	Note that the above isomorphism is independent of $\lambda$ by Proposition \ref{t:HPi unique}. 
\end{secnumber}

\section{Local Drinfeld's lemma over polyannuli}

We assume $(k,\sigma)=(\mathbb{F}_q,\id_{\mathscr{O}_K})$. 
For a connected subinterval $I$ of $(0,\infty)$, we denote by $A_K[I]$ the rigid annulus $|t|\in I$ over $K$. 
Recall that the Robba ring $\mathcal{R}$ over $K$ is defined as
\[
\mathcal{R}:=\lim_{\varepsilon\to 1} \mathscr{O}(A_K[\varepsilon,1[).
\]
Let $n$ be a positive integer. We consider the $n$-fold Robba ring $\mathcal{R}^n$ over $K$:
\begin{eqnarray}
	\mathcal{R}^n &:=& \lim_{\varepsilon_1,\ldots,\varepsilon_n\to 1} \mathscr{O}(\prod_{i=1}^n A_K[\varepsilon_i,1[)\nonumber \label{eq:Robba-n}\\
	&=& \{\sum_{I=(i_1,\ldots,i_n)} a_I \underline{t}^I ;~ |a_I|\underline{\rho}^I \to 0 ~(|i_1|+\cdots+|i_n|\to \infty), \forall \rho_i\in [\varepsilon_i,1[ \textnormal{ for some } \varepsilon_i\in ]0,1[ \}, \nonumber 
\end{eqnarray}
where $\underline{t}^I:=t_1^{i_1}\cdots t_n^{i_n}$ and $\rho^I:=\rho_1^{i_1}\cdots\rho_n^{i_n}$. 
For any algebraic extension $L/K$, we set $\mathcal{R}^n_{L}:=\mathcal{R}^n\otimes_{K}L$. 

Let $h$ be an integer $\ge 1$, and $\sigma$ a continuous automorphism of $\bQp$ lifting the $h$-th Frobenius on $\overline{k}$ and preserving each finite extension $L$ of $K$. 
We denote by $\phi_i\colon \mathcal{R}^n\to \mathcal{R}^n$ the $K$-linear endomorphism defined by $t_i\mapsto t_i^{p^h}$ and $t_j\mapsto t_j$ if $j\neq i$. 

We consider the category $\MIC(\mathcal{R}^n/K)$ of free $\mathcal{R}^n$-modules of finite rank equipped with an integrable $K$-linear connection and we denote by $\MIC^{\uni}(\mathcal{R}^n/K)$ its full subcategory consisting of unipotent objects, which is isomorphic to a successive extension of the trivial connection over $\mathcal{R}^n$. 

\textit{A partial Frobenius structure of order $h$ (resp. a partial Frobenius structure)} on an object $(M,\nabla)$ of $\MIC(\mathcal{R}^n/K)$ consists of isomorphisms $\varphi_i\colon (\sigma\circ\phi_i)^*(M,\nabla)\simeq (M,\nabla)$ of $\MIC(\mathcal{R}^n/K)$ commuting with each other (resp. without specifying order $h$). 
We denote by $\MIC(\mathcal{R}^n/K,\Phi)$ the full subcategory of $\MIC(\mathcal{R}^n/K)$ consisting of objects whose irreducible sub-quotients can be equipped with a partial Frobenius structure (of order $h'$ with $h|h'$). 
Note that a unipotent object lies in $\MIC(\mathcal{R}^n/K,\Phi)$. 

Let $K^{\ur}$ be the maximal unramified extension of $K$ in $\bQp$. 
We can extend above definition to free $\mathcal{R}^n_{K^{\ur}}$-modules of finite rank with an integrable $K^{\ur}$-linear connection. 
Note that an object of $\MIC(\mathcal{R}^n_{K^{\ur}}/K^{\ur})$ comes from the extension of scalars of an object of $\MIC(\mathcal{R}^n_{L}/L)$ for an unramified finite extension $L/K$. 

Let $\mathcal{K}:=k( (t))$, $G_{\mathcal{K}}:=\Gal(\mathcal{K}^{\sep}/\mathcal{K})$, and $I_{\mathcal{K}}$ the inertia subgroup of the Galois group $G_{\mathcal{K}}$. 
We reformulate the local monodromy theorem for $\MIC(\mathcal{R}^n/K,\Phi)$ \cite[Theorem 3.3.6]{kedlaya-mono-rel} as follows; this  generalizes Andr\'e's result \cite[Th\'eor\`eme 7.1.1]{And02II} to polyannuli.  

\begin{theorem}
	\label{t:local Drinfeld lemma}
	The category $\MIC(\mathcal{R}^n_{K^{\ur}}/K^{\ur},\Phi)$ is a neutral Tannakian category over $K^{\ur}$ and its Tannakian group is isomorphic to $(I_{\mathcal{K}}\times \Ga)^n$.
\end{theorem}

\subsection{Unipotent connections on polyannuli}

\begin{prop} \label{p:unipotent}
	The functor $(V,N_1,\ldots,N_n) \mapsto (V\otimes_K \mathcal{R},\nabla_{\mathcal{N}})$, where the connection $\nabla_{\mathcal{N}}$ is defined by:
	\[
	\nabla_{\mathcal{N}}(v \otimes 1):= \sum_{i=1}^n N_iv\otimes \frac{dt_i}{t_i},
	\]
	induces an equivalence of tensor categories between the category of finite dimensional $K$-vector spaces with $n$ commuting nilpotent operators $\mathcal{N}=(N_1,\ldots,N_n)$ and the category $\MIC^{\uni}(\mathcal{R}^n/K)$. 
\end{prop}

\begin{proof}
	We construct a quasi-inverse of the above functor. Let $(M,\nabla)$ be a unipotent connection on a polyannulus $\prod_{i=1}^n A_K[\varepsilon_i,1[$. 
	The operators $\partial_i:=\nabla(t_i \frac{d}{d t_i})$ on $M$ commute with each other; set $\partial:=\partial_1\circ\cdots \circ \partial_n$. 
	Since $(M,\nabla)$ is unipotent, the $K$-vector space $V:=\cup_{n\ge 1} (\Ker \partial)^n$ has the same rank as $M$. 
	We claim that the operator $\partial_i$ is nilpotent on $V$. 
	Indeed, if $x$ is a point of the polyannulus $\prod_{j\neq i} A_K[\varepsilon_j,1[$, then its fiber $(M_x,\nabla_x)$ is a unipotent connection over the annulus $A_{K(x)}[\varepsilon_i,1[$. 
	We deduce that $\partial_i$ is nilpotent on $V\otimes_K K(x)$ and the claim follows. 
	Then
	\[
	(M,\nabla)\mapsto (\cup_{n\ge 1} (\Ker \partial)^n, \partial_1,\ldots,\partial_n) 
	\]
	defines a quasi-inverse functor. The proposition follows. 
\end{proof}

\begin{remark}
	Unlike the case of an annulus \cite[Lemma 4.3]{Mat02}, a unipotent object on a polyannulus may not admit a Frobenius structure. 
\end{remark}

\begin{secnumber} \textit{Proof of Theorem \ref{t:local Drinfeld lemma}.}
	Let $(M,\nabla)$ be an object of $\MIC(\mathcal{R}^n_{K^{\ur}}/K^{\ur},\Phi)$ defined on a polyannulus $\prod_{i=1}^n A_{K^{\ur}}[\varepsilon_i,1[$.  
	By \cite[Theorem 3.3.6]{kedlaya-mono-rel}, there exist eligible \'etale covers $\{Y_i\to A_{K^{\ur}}[\varepsilon_i,1[\}_{i=1}^n$ such that the pullback of $(M,\nabla)$ to $\prod_{i=1}^n Y_i$ is unipotent. 
	Then by Proposition \ref{p:unipotent}, we obtain a fiber functor:
	\[
	\omega\colon \MIC(\mathcal{R}^n_{K^{\ur}}/K^{\ur},\Phi)\to \Vect_{K^{\ur}},
	\]
	by forgetting nilpotent operators. 
	This makes $\MIC(\mathcal{R}^n_{K^{\ur}}/K^{\ur},\Phi)$ into a neutral Tannakian category over $K^{\ur}$. 
	In this way, we obtain a continuous action of $I_{\mathcal{K}}^n$ on $\omega(M,\nabla)$ via $I_{\mathcal{K}}^n\to \prod \Aut(Y_i/A_{K^{\ur}}[I])$, which commutes with the action of nilpotent operators. 
	Then the theorem follows. \hfill \qed 
\end{secnumber}

\subsection{Construction of Weil--Deligne representations}

In this subsection, we briefly review representations of a self-product of the Weil group of $\mathcal{K}$ associated to differential modules with a partial Frobenius structure on polyannuli, following \cite{Mar08}. 

\begin{secnumber}
	Recall that $q=p^s$ is the cardinality of $k$. 
	We denote by $\PhiMIC(\mathcal{R}^n/K)$ the category of free modules over $\mathcal{R}^n$ equipped with an integrable $K$-linear connection and a partial Frobenius structure of order $s$. 

	We denote by $\Del_{K^{\ur}}(G_{\mathcal{K}}^n)$ the category of triples $(V,\varphi_1,\ldots,\varphi_n,N_1,\ldots,N_n)$ consisting of a continuous semilinear representation of $G_{\mathcal{K}}^n$ on a finite dimensional $K^{\ur}$-vector space $V$ (with discrete topology), $\sigma$-semilinear equivariant Frobenius isomorphisms $\{\varphi_i\}_{i=1}^n$ on $V$ commuting with each other, and equivariant monodromy operators $\{N_i\colon V\to V\}_{i=1}^n$ commuting with each other satisfying $N_i\varphi_j=q^{\delta_{ij}}\varphi_jN_i$ for $i,j=1,\ldots,n$. 

	The proof of Theorem \ref{t:local Drinfeld lemma} shows that there exists a canonical tensor functor:
	\[
		\PhiMIC(\mathcal{R}^n/K) \to \Del_{K^{\ur}}(G_{\mathcal{K}}^n).
	\]
\end{secnumber}

\begin{secnumber}
	Let $W_{\mathcal{K}}$ be the Weil group of $\mathcal{K}$ and $v_i\colon W_{\mathcal{K}}^n\to \mathbb{Z}$ the projection on the $i$-th component. 
	We consider the category $\Rep_{K^{\ur}}(\WD^n_{\mathcal{K}})$ of pairs $(V,N_1,\cdots,N_n)$ consisting of a continuous linear representation $\rho$ of $W_{\mathcal{K}}^n$ on  a finite dimensional $K^{\ur}$-vector space $V$ (with discrete topology), monodromy operators $\{N_i\}_{i=1}^n$ commuting with each other satisfying $N_i\rho(g)=q^{v_i(g)}\rho(g)N_i$ for $i=1,\ldots,n$, and $g\in W_{\mathcal{K}}^n$.
	When $n=1$, this is the category of Weil--Deligne representations of $W_{\mathcal{K}}$. 

	We have a Frobenius linearization functor
	\[
		L\colon \Del_{K^{\ur}}(G_{\mathcal{K}}^n)\to \Rep_{K^{\ur}}(\WD^n_{\mathcal{K}}),
	\]
	sending $(V,\varphi_i,N_i)$ to a continuous linear representation $\rho$ of $W_{\mathcal{K}}^n$ on $V$ defined by 
	\[
	\rho(g)(m):=g(\prod_{i=1}^n \varphi^{v_i(g)}_i(m)),
	\]
	together with monodromy operators $\{N_i\}_{i=1}^n$. 
	In summary, we obtain a tensor functor:
	\[
	\PhiMIC(\mathcal{R}^n/K) \to \Rep_{K^{\ur}}(\WD^n_{\mathcal{K}}).
	\]
\end{secnumber}

\end{document}